\documentclass[a4paper]{article}
\usepackage[margin=25mm]{geometry}
\usepackage{amsmath}
\usepackage{amssymb}
\usepackage{amsthm}
\usepackage{bbm}
\usepackage{braket}
\usepackage{titlesec}
\usepackage{titlefoot}
\usepackage{bm}
\usepackage{color}
\usepackage{comment}
\titleformat*{\section}{\normalsize\bfseries}
\titleformat*{\subsection}{\normalsize\bfseries}
\allowdisplaybreaks
\def\R{\mathbb{R}}
\def\e{{\varepsilon}}        

\def\p{\partial}

\newtheorem{thm}{Theorem}[section]
\newtheorem{lem}[thm]{Lemma}
\newtheorem{cor}[thm]{Corollary}
\newtheorem{prop}[thm]{Proposition}
\newtheorem{rem}[thm]{Remark}

\makeatletter
\@addtoreset{equation}{section}
 
  \makeatother
 
\begin{document}
%%%%%%%%%%%%%%%%%%%%%%%%%%%%%%%%%%%%%%%%%%%%%%%%%%%%

%%%%%%%%%%%%%%%%%%%%%%%%%%%%%%%%%%%%%%%%%%%%%%%%%%%%
\title{
\vspace{-1cm}
\large{\bf ASYMPTOTIC PROFILES OF SOLUTIONS FOR THE GENERALIZED FORNBERG--WHITHAM EQUATION WITH DISSIPATION}}
\author{Ikki Fukuda\\ [.7em]
Faculty of Engineering, Shinshu University}
\date{}
\maketitle
%%%%%%%%%%%%%%%%%%%%%%%%%%%%%%%%%%%%%%%%%%%%%%%%%%%%

%%%%%%%%%%%%%%%%%%%%%%%%%%%%%%%%%%%%%%%%%%%%%%%%%%%%
\footnote[0]{2020 Mathematics Subject Classification: 35B40, 35Q53.}
%%%%%%%%%%%%%%%%%%%%%%%%%%%%%%%%%%%%%%%%%%%%%%%%%%%%

\vspace{-0.75cm}
%%%%%%%%%%%%%%%%%%%%%%%%%%%%%%%%%%%%%%%%%%%%%%%%%%%%
\begin{abstract}
We consider the Cauchy problem for the generalized Fornberg--Whitham equation
with dissipation. This is one of the nonlinear, nonlocal and dispersive-dissipative
equations. The main topic of this paper is an asymptotic analysis for the
solutions to this problem. We prove that the solution to this problem converges
to the modified heat kernel. Moreover, we construct the second term of
asymptotics for the solutions depending on the degree of the nonlinearity.
In view of those second asymptotic profiles, we investigate the effects
of the dispersion, dissipation and nonlinear terms on the asymptotic behavior
of the solutions.
\end{abstract}
%%%%%%%%%%%%%%%%%%%%%%%%%%%%%%%%%%%%%%%%%%%%%%%%%%%%

\medskip
%%%%%%%%%%%%%%%%%%%%%%%%%%%%%%%%%%%%%%%%%%%%%%%%%%%%
\noindent
{\bf Keywords:} 
Fornberg--Whitham equation with dissipation; Cauchy problem; Asymptotic profiles. 
%%%%%%%%%%%%%%%%%%%%%%%%%%%%%%%%%%%%%%%%%%%%%%%%%%%%

%%%%%%%%%%%%%%%%%%%
\section{Introduction}
%%%%%%%%%%%%%%%%%%%

We consider the Cauchy problem for the following integro-differential equation:
\begin{align}\label{VFW}
\begin{split}
&u_{t}+(|u|^{p-1}u)_{x}+\int_{\R}Be^{-b|x-y|}u_{y}(y, t)dy=\mu u_{xx}, \ \ x\in \R, \ t>0,\\
&u(x, 0)=u_{0}(x) , \ \ x\in \R, 
\end{split}
\end{align}
where $u=u(x, t)$ is a real-valued unknown function, $u_{0}(x)$ is a given initial data, $p>2$ and $B, b, \mu>0$. 
The subscripts $t$ and $x$ denote the partial derivatives with respect to $t$ and $x$, respectively. 
This equation is one of a model for nonlinear waves taking into account the dispersive-dissipative processes as well as the convection effects. 
The purpose of this study is to analyze the large time asymptotic behavior of the solutions to \eqref{VFW}. 
Based on that analysis, we would like to investigate the effects of the dispersion, dissipation and nonlinear terms on the asymptotic profiles of the solutions. 

First of all, let us explain about the original problem and background of \eqref{VFW}.
If we take $\mu=0$ and replace $(|u|^{p-1}u)_{x}$ with $\beta uu_{x}$ ($\beta\neq0$) in \eqref{VFW}, we obtain the Fornberg--Whitham equation: 
\begin{align}\label{FW}
\begin{split}
&u_{t}+\beta uu_{x}+\int_{\R}Be^{-b|x-y|}u_{y}(y, t)dy=0, \ \ x\in \R, \ t>0, \\
&u(x, 0)=u_{0}(x) , \ \ x\in \R. 
\end{split}
\end{align}
The above equation~\eqref{FW} was derived by Whitham \cite{W67} and Fornberg--Whitham \cite{FW78} in the late 1900s, as a mathematical model for so-called ``breaking waves". Here, roughly speaking, wave-breaking means blow-up of the spatial derivative of the solution, i.e. $\limsup_{t\uparrow T_{0}}\|u_{x}(\cdot, t)\|_{L^{\infty}}=\infty$ for some $T_{0}>0$. Wave-breaking phenomena for equations with the nonlocal dispersion term, was first studied by Seliger \cite{S68}. He studied the following more general equation called the Whitham equation: 
\begin{align}\label{W-type}
\begin{split}
&u_{t}+\beta uu_{x}+\int_{\R}K(x-y)u_{y}(y, t)dy=0, \ \ x\in \R, \ t>0, \\
&u(x, 0)=u_{0}(x) , \ \ x\in \R, 
\end{split}
\end{align}
where $K(x)$ is a given real even function. In \cite{S68}, he presented that wave-breaking is possible for \eqref{W-type}, by a formal argument. 
Also, the wave-breaking phenomena for solutions to \eqref{W-type} with a regular kernel $K(x)=Be^{-b|x|}$ like \eqref{FW} was studied in \cite{W74}. 
For another perspective on \eqref{W-type}, see e.g. \cite{G78}. In addition, for the related results about more general nonlocal dispersive equations, we can also refer to \cite{NS94}. 

A mathematically rigorous analysis of the wave-breaking phenomena for equations with the nonlocal dispersion term such as \eqref{FW} was first performed by Constantin--Escher \cite{CE98}. 
They gave a sufficient condition for the blow-up of solutions to \eqref{W-type}. After that, their result was improved by Ma--Liu--Qu in \cite{MLQ16}. 
Moreover, Haziot \cite{H17} obtained another blow-up condition for \eqref{FW} with $\beta=1$, which includes only parameter $B$ and does not include $b$. Recently, Itasaka \cite{I21} proposed a new blow-up condition for~\eqref{FW} with $\beta=1$, which includes the both parameters $(B, b)$. 
In addition to the blow-up result, he investigated some relations between the Fornberg--Whitham equation \eqref{FW} and the inviscid Burgers equation.

Next, let us introduce some results related to numerical analysis. Tanaka \cite{T13} and H\"{o}rmann--Okamoto \cite{HO19} studied \eqref{FW} numerically. 
Their results suggested that \eqref{FW} has blow-up solutions and global solutions depending on the initial data $u_{0}(x)$ and the parameters $(B, b)$. 
As we mentioned in the above paragraph, mathematically rigorous blow-up conditions for \eqref{FW} have been studied by many researchers. 
On the other hand, we have not had any mathematical result of the existence of global solutions for \eqref{FW} yet. 
As a well known fact, for suitable regular initial data, the KdV equation always has global solutions. %(cf.~\cite{CKSTT03, G09, KPV96, Ki09}). 
This is because the nonlinear effect and the dispersive effect balance each other, and then the energy is conserved. From this perspective, to show the global existence of solutions to \eqref{FW}, it would be effective to investigate some relationship between the nonlinear effect and the dispersive effect in \eqref{FW} and compare the dispersion term with other type ones. As related works, the author and Itasaka \cite{FI21} studied \eqref{VFW1} and \eqref{KdVB} below, based on this consideration. We will explain it in the next paragraph. 

In order to explain about the motivation for our study, 
we shall introduce a known result for the Fornberg--Whitham equation \eqref{FW} with the dissipation term $\mu u_{xx}$. 
If we replace $(|u|^{p-1}u)_{x}$ with $\beta uu_{x}$ ($\beta\neq0$) in \eqref{VFW}, we obtain the following problem: 
\begin{align}\label{VFW1}
\begin{split}
&u_{t}+\beta uu_{x}+\int_{\R}Be^{-b|x-y|}u_{y}(y, t)dy=\mu u_{xx}, \ \ x\in \R, \ t>0,\\
&u(x, 0)=u_{0}(x) , \ \ x\in \R. 
\end{split}
\end{align}
For the above problem \eqref{VFW1}, it is easy to verify that the solution exists globally in time due to the dissipation effect. Namely, it can be said that \eqref{VFW1} is easier than \eqref{FW} to investigate the structure of the nonlocal dispersion term and its interaction with the nonlinear term. 
Therefore, by studying \eqref{VFW1} instead of \eqref{FW}, we can expect to get some hints for analyzing \eqref{FW}. 
From this point of view, the author and Itasaka \cite{FI21} studied \eqref{VFW1}, and obtained the asymptotic profiles of the solutions. Actually, the solution of \eqref{VFW1} converges to the nonlinear diffusion wave which is a modification of the self-similar solution to the following Burgers equation:
\[
\chi_{t}+\frac{2B}{b}\chi_{x}+\beta \chi \chi_{x}=\mu \chi_{xx}, \ \ x\in \R, \ t>0. 
\]
Here, we note that $\chi(x, t)$ is defined by 
\begin{equation*}
\chi(x, t):=\frac{1}{\sqrt{1+t}} \chi_{*} \left(\frac{x-\frac{2B}{b}(1+t)}{\sqrt{1+t}}\right), \ \ x\in \R, \ t>0,
\end{equation*}
where
\begin{equation*}
\chi_{*}(x):=\frac{\sqrt{\mu}}{\beta}\frac{\left(e^{\frac{\beta M}{2\mu}}-1\right)e^{-\frac{x^{2}}{4\mu}}}{\sqrt{\pi}+\left(e^{\frac{\beta M}{2\mu}}-1\right)\int_{x/\sqrt{4\mu}}^{\infty}e^{-y^{2}}dy}, \ \ M:=\int_{\R}u_{0}(x)dx. 
\end{equation*}
More precisely, if $u_{0} \in H^{2}(\R)\cap L^{1}(\R)$, $xu_{0}\in L^{1}(\R)$, $M\neq0$ and $\|u_{0}\|_{H^{2}}+\|u_{0}\|_{L^{1}}$ is sufficiently small, then the solution to \eqref{VFW1} satisfies the following optimal decay estimate: 
\begin{equation*}
\left\|u(\cdot, t)-\chi(\cdot, t)\right\|_{L^{q}}=\left(C_{0}+o(1)\right)(1+t)^{-\frac{1}{2}\left(1-\frac{1}{q}\right)-\frac{1}{2}}\log(1+t) \ \ \text{as} \ \ t\to \infty, 
\end{equation*}
for any $2\le q\le \infty$, where $C_{0}=C_{0}(B, b, \beta, \mu)$ is a certain positive constant. 

Moreover, in \cite{FI21}, the second and the third order asymptotic profiles of the solutions to \eqref{VFW1} also have been obtained. 
Here, we remark that the similar results of them were also obtained for other Burgers type equations such as the KdV--Burgers equation in \cite{HN06, KP05} and the BBM--Burgers equation in \cite{HKN07}. 
In addition, the author and Itasaka also mentioned in \cite{FI21} that the effect of the nonlocal dispersion term on the more higher-order asymptotic profiles. 
Furthermore, they compared the results for \eqref{VFW1} with the ones of the following KdV--Burgers equation: 
\begin{align}\label{KdVB}
\begin{split}
&u_{t}+\frac{2B}{b}u_{x}+\beta uu_{x}+\frac{2B}{b^{3}}u_{xxx}=\mu u_{xx}, \ \ x\in \R, \ t>0,\\
&u(x, 0)=u_{0}(x) , \ \ x\in \R. 
\end{split}
\end{align}
Roughly speaking, the results given in \cite{FI21} suggest that the difference between \eqref{VFW1} and \eqref{KdVB} appears from the third order asymptotic profiles of the solutions (for details see \cite{FI21}). 

On the other hand, compared with \eqref{VFW1}, our target equation \eqref{VFW} has the more general nonlinearity $(|u|^{p-1}u)_{x}$. 
In particular, for $p>2$, the nonlinearity seems weak because $(|u|^{p-1}u)_{x}$ decays faster than $uu_{x}$ as the solution decays. 
For this reason, the asymptotic profile of the solution to \eqref{VFW} is expected to be different from that given in \cite{FI21}.
Therefore, it is worthwhile to study \eqref{VFW} to investigate the structure of the solution and some interaction between nonlinear and dispersion effects. In order to obtain the asymptotic profile for the solution to \eqref{VFW}, we develop the method used in \cite{K99} for the Cauchy problem of some dispersive-dissipative equations, such as the following generalized KdV--Burgers equation: 
\begin{align*}
&u_{t}-\mu u_{xx}+u_{xxx}+(|u|^{p-1}u)_{x}=0, \ \ x\in \R, \ t>0, \\
&u(x, 0)=u_{0}(x) , \ \ x\in \R. 
\end{align*}

\par\noindent
\textbf{\bf{Main Result.}} Now, let us state our main result: 
%%%%%%%%%%%%%%%%%%%%%%%%%%%%%%%%%%%%%%%%%%%%%%%%%%%%
\begin{thm}\label{thm.main}
Let $p>2$. Assume that $u_{0} \in H^{1}(\R) \cap L^{1}(\R)$ and $E_{0}:=\|u_{0}\|_{H^{1}}+\|u_{0}\|_{L^{1}}$ is sufficiently small. Then, \eqref{VFW} has a unique global mild solution $u \in C([0, \infty); H^{1}(\R))$ satisfying 
\begin{equation}\label{u-decay} 
\left\|\p_{x}^{l}u(\cdot, t)\right\|_{L^{2}}\le CE_{0}(1+t)^{-\frac{1}{4}-\frac{l}{2}}, \ \ t\ge0, \ l=0, 1. 
\end{equation}
Moreover, the solution $u(x, t)$ satisfies the following estimate: 
\begin{equation}\label{u-decay-Lq}
\left\|u(\cdot, t)\right\|_{L^{q}}\le CE_{0}(1+t)^{-\frac{1}{2}\left(1-\frac{1}{q}\right)}, \ \ t\ge0, \ 2\le q\le \infty. 
\end{equation}
Furthermore, if $xu_{0} \in L^{1}(\R)$, then the solution $u(x, t)$ satisfies the following asymptotics: 
\begin{align}
&\lim_{t\to \infty}t^{\frac{1}{2}\left(1-\frac{1}{q}\right)+\frac{p-2}{2}}\left\|u(\cdot, t)-MG_{0}(\cdot, t)+\left(|M|^{p-1}M\right)W_{p}(\cdot, t)\right\|_{L^{q}}=0, \ \ 2<p<3, \label{u-asymp-2<p<3}\\
&\lim_{t\to \infty}\frac{t^{\frac{1}{2}\left(1-\frac{1}{q}\right)+\frac{1}{2}}}{\log t}\left\|u(\cdot, t)-MG_{0}(\cdot, t)+\frac{M^{3}}{4\sqrt{3}\pi\mu}(\log t)\p_{x}G_{0}(\cdot, t)\right\|_{L^{q}}=0, \ \ p=3, \label{u-asymp-p=3}\\
&\lim_{t\to \infty}t^{\frac{1}{2}\left(1-\frac{1}{q}\right)+\frac{1}{2}}\left\|u(\cdot, t)-MG_{0}(\cdot, t)+\left(m+\mathcal{M}\right)\p_{x}G_{0}(\cdot, t)+\frac{2BM}{b^{3}}t\p_{x}^{3}G_{0}(\cdot, t)\right\|_{L^{q}}=0, \ \ p>3,  \label{u-asymp-p>3} 
\end{align}
for any $2\le q \le \infty$. Here, $G_{0}(x, t)$ and $W_{p}(x, t)$ are defined as follows: 
\begin{equation}\label{asymp-func}
G_{0}(x, t):=G\left(x-\frac{2B}{b}t, t\right), \ \  
W_{p}(x, t):=t^{-\frac{p-1}{2}}w_{p}\left(\frac{x-\frac{2B}{b}t}{\sqrt{t}}\right), \ \ x\in \R, \ t>0, 
\end{equation}
where $G(x, t)$ and $w_{p}(x)$ are defined by 
\begin{equation}\label{kernel}
G(x, t):=\frac{1}{\sqrt{4\pi \mu t}}\exp\left({-\frac{x^{2}}{4\mu t}}\right), \ \ 
w_{p}(x):=\frac{d}{dx}\left(\int_{0}^{1}\left(G(1-s)*G^{p}(s)\right)(x)ds\right). 
\end{equation}
Also, we define the constants $M$, $m$ and $\mathcal{M}$ as follows: 
\begin{equation}\label{M-m-notation}
M:=\int_{\R}u_{0}(x)dx, \ \  m:=\int_{\R}xu_{0}(x)dx, \ \ \mathcal{M}:=\int_{0}^{\infty}\int_{\R}(|u|^{p-1}u)(y, \tau)dyd\tau \ \ \text{for} \ \ p>3. 
\end{equation}
\end{thm}
%%%%%%%%%%%%%%%%%%%%%%%%%%%%%%%%%%%%%%%%%%%%%%%%%%%%

%%%%%%%%%%%%%%%%%%%%%%%%%%%%%%%%%%%%%%%%%%%%%%%%%%%%
\begin{rem}
{\rm
From the above result, we see that the first asymptotic profile of the solution to \eqref{VFW} is given by $MG_{0}(x, t)$ for all $p>2$. 
As we can see from the shape of $MG_{0}(x, t)$, the strongest effect is the dissipation effect and the dispersion term acts as a convection effect. 
This means that the dispersion effect is not so strong in the first term of the asymptotics. 
A result analogous to this conclusion has been obtained for the more general dispersive-dissipative nonlinear equation in Theorem~4.13 of ~\cite{HKNS06}. On the other hand, the present result provides the more specific asymptotic formula for the solution, including the second asymptotic profiles.
}
\end{rem}
%%%%%%%%%%%%%%%%%%%%%%%%%%%%%%%%%%%%%%%%%%%%%%%%%%%%

%%%%%%%%%%%%%%%%%%%%%%%%%%%%%%%%%%%%%%%%%%%%%%%%%%%%
\begin{rem}
{\rm
The second asymptotic profiles of the solution are divided into three cases depending on $p$. 
In particular, note that for $2<p\le3$, the effect of the dispersion term does not appear strongly in the second asymptotic profiles as well as in the first asymptotic profile. 
This can be said to indicate that the effect of the nonlinearity is stronger than the effect of the dispersion term. 
On the other hand, note that for $p>3$, the second asymptotic profile contains all of the dissipation, dispersion, and nonlinear effects.
}
\end{rem}
%%%%%%%%%%%%%%%%%%%%%%%%%%%%%%%%%%%%%%%%%%%%%%%%%%%%

%%%%%%%%%%%%%%%%%%%%%%%%%%%%%%%%%%%%%%%%%%%%%%%%%%%%
\begin{rem}
{\rm
Some similar results for \eqref{u-asymp-2<p<3}, \eqref{u-asymp-p=3} and \eqref{u-asymp-p>3} and other related results are also obtained for several dissipative type equations, such as the generalized KdV--Burgers equation \cite{K97, K99}, the generalized BBM--Burgers equation \cite{K99, PZ02} and also the convection-diffusion equation \cite{Z93}. 
}
\end{rem}
%%%%%%%%%%%%%%%%%%%%%%%%%%%%%%%%%%%%%%%%%%%%%%%%%%%%

%%%%%%%%%%%%%%%%%%%%%%%%%%%%%%%%%%%%%%%%%%%%%%%%%%%%
\begin{rem}
{\rm
In the case of $p=3$, we can actually prove the more stronger result than \eqref{u-asymp-p=3}. 
More precisely, the more improved asymptotic rate can be obtained. 
For details, see Theorem~\ref{thm.nl-p=3} below.
}
\end{rem}
%%%%%%%%%%%%%%%%%%%%%%%%%%%%%%%%%%%%%%%%%%%%%%%%%%%%

Moreover, in view of the second asymptotic profiles, we are able to obtain the optimal asymptotic rates to the modified heat kernel $G_{0}(x, t)$ as follows: 
%%%%%%%%%%%%%%%%%%%%%%%%%%%%%%%%%%%%%%%%%%%%%%%%%%%%
\begin{cor}
Under the same assumptions in Theorem~\ref{thm.main}, we have the following estimate: 
\begin{equation*}
\left\|u(\cdot, t)-MG_{0}(\cdot, t)\right\|_{L^{q}}=
\begin{cases}
\displaystyle \left(|M|^{p}\left\|w_{p}\right\|_{L^{q}}+o(1)\right)t^{-\frac{1}{2}\left(1-\frac{1}{q}\right)-\frac{p-2}{2}},&\ 2<p<3, \\[0.5em]
\displaystyle \left(\frac{|M|^{3}}{4\sqrt{3}\pi \mu}\left\|\p_{x}G(\cdot, 1)\right\|_{L^{q}}+o(1)\right)t^{-\frac{1}{2}\left(1-\frac{1}{q}\right)-\frac{1}{2}}\log t,&\ p=3, \\[1em]
\displaystyle \left(\left\|(m+\mathcal{M})\p_{x}G(\cdot, 1)+\frac{2BM}{b^{3}}\p_{x}^{3}G(\cdot, 1)\right\|_{L^{q}}+o(1)\right)t^{-\frac{1}{2}\left(1-\frac{1}{q}\right)-\frac{1}{2}},&\ p>3,  
\end{cases}
\end{equation*}
as $t\to \infty$, for any $2\le q\le \infty$.  
\end{cor}
%%%%%%%%%%%%%%%%%%%%%%%%%%%%%%%%%%%%%%%%%%%%%%%%%%%%

The rest of this paper is organized as follows. 
First, we prove the global existence and the time decay estimates for the solutions to \eqref{VFW} in Section~2. 
Next, in Section~3, we introduce some auxiliary lemmas and propositions to prove the main result. 
Finally, we give the proof of our main result Theorem~\ref{thm.main} in Section~4. 
This section is divided into three subsections. Subsection~4.1 is for $2<p<3$, Subsection~4.2 is for $p=3$ and Subsection~4.3 is for $p>3$. 
The main difficulty of the proof of Theorem~\ref{thm.main}, especially the proof of \eqref{u-asymp-p>3}, 
is how to treat the nonlocal dispersion term $\int_{\R}Be^{-b|x-y|}u_{y}(y, t)dy$. 
To overcome this difficulty, we transform this term to $\frac{2B}{b}\p_{x}u+\frac{2B}{b^{3}}\p_{x}^{3}u+\frac{2B}{b^{3}}(b^{2}-\p_{x}^{2})^{-1}\p_{x}^{5}u$ via the Fourier transform, and apply the idea of the asymptotic analysis for the generalized KdV--Burgers equation used in \cite{K99}. 

\medskip
\par\noindent
\textbf{\bf{Notations.}} For $1\le p \le \infty$, $L^{p}(\R)$ denotes the usual Lebesgue spaces. 
Then, for $f, g \in L^{1}(\R)\cap L^{2}(\R)$, we denote the Fourier transform of $f$ and the inverse Fourier transform of $g$ as follows:
\begin{align*}
\hat{f}(\xi):=\mathcal{F}[f](\xi)=\frac{1}{\sqrt{2\pi}}\int_{\R}e^{-ix\xi}f(x)dx, \quad \mathcal{F}^{-1}[g](x):=\frac{1}{\sqrt{2\pi}}\int_{\R}e^{ix\xi}g(\xi)d\xi.
\end{align*}
Also, for $s\ge0$, we define the Sobolev spaces by 
\begin{equation*}
H^{s}(\R):=\left\{f\in L^{2}(\R); \ \left\|f\right\|_{H^{s}}:=\left(\int_{\R}\left(1+\xi^{2}\right)^{s}|\hat{f}(\xi)|^{2}d\xi \right)^{\frac{1}{2}}<\infty \right\}. 
\end{equation*}
Moreover, we denote $C([0, \infty); H^{s}(\R))$ as the space of $H^{s}$-valued continuous functions on $[0, \infty)$.
\smallskip

Throughout this paper, $C$ denotes various positive constants, which may vary from line to line during computations. Also, it may depend on the norm of the initial data. However, we note that it does not depend on the space variable $x$ and the time variable $t$. 

%%%%%%%%%%%%%%%%%%%
\section{Global Existence and Decay Estimates}  
%%%%%%%%%%%%%%%%%%%

In this section, we would like to show the global existence and the decay estimates \eqref{u-decay} and \eqref{u-decay-Lq} of the solutions to~\eqref{VFW}. 
To discuss them, we consider the following integral equation associated with \eqref{VFW}:
\begin{equation}\label{integral-eq}
u(t)=T(t)*u_{0}-\int_{0}^{t}\p_{x}T(t-\tau)*(|u|^{p-1}u)(\tau)d\tau,  
\end{equation}
where the integral kernel $T(x, t)$ is defined by 
\begin{equation}\label{T-linear}
T(x, t):=\frac{1}{\sqrt{2\pi}}\mathcal{F}^{-1}\left[\exp\left(-\mu t\xi^{2}-\frac{i2Bbt\xi}{b^{2}+\xi^{2}}\right)\right](x).
\end{equation}
For this function, we note that the following estimate holds. The proof is the same as Lemma 2.2 in \cite{F19}.
%%%%%%%%%%%%%%%%%%%%%%%%%%%%%%%%%%%%%%%%%%%%%%%%%%%%
\begin{lem}\label{lem.linear-L2}
Let $s$ be a non-negative integer. Suppose $f\in H^{s}(\R) \cap L^{1}(\R)$. Then, the estimate
\begin{equation}\label{linear-L2}
\left\| \p^{l}_{x}(T(t)*f)\right\|_{L^{2}} \le C(1+t)^{-\frac{1}{4}-\frac{l}{2}}\left\|f\right\|_{L^{1}}+Ce^{-\mu t}\left\| \p^{l}_{x}f\right\|_{L^{2}}, \ \ t\ge0
\end{equation}
holds for any integer $l$ satisfying $0\le l \le s$.
\end{lem}
%%%%%%%%%%%%%%%%%%%%%%%%%%%%%%%%%%%%%%%%%%%%%%%%%%%%
Now, let us prove the global existence and the decay estimates \eqref{u-decay} and \eqref{u-decay-Lq} of the solutions to~\eqref{VFW}. 
We give the proof of them by slightly modifying the method used in Theorem~2.2 of \cite{FI21}. 
%%%%%%%%%%%%%%%%%%%%%%%%%%%%%%%%%%%%%%%%%%%%%%%%%%%%
\begin{prop}\label{thm.GE-decay}
Let $p\ge2$. Assume that $u_{0} \in H^{1}(\R) \cap L^{1}(\R)$ and $E_{0}=\|u_{0}\|_{H^{1}}+\|u_{0}\|_{L^{1}}$ is sufficiently small. Then, \eqref{VFW} has a unique global mild solution $u \in C([0, \infty); H^{1}(\R))$ satisfying 
\begin{equation}\tag{\ref{u-decay}}
\left\|\p_{x}^{l}u(\cdot, t)\right\|_{L^{2}}\le CE_{0}(1+t)^{-\frac{1}{4}-\frac{l}{2}}, \ \ t\ge0, \ l=0, 1. 
\end{equation}
Moreover, the solution satisfies the following estimate: 
\begin{equation}\tag{\ref{u-decay-Lq}}
\left\|u(\cdot, t)\right\|_{L^{q}}\le CE_{0}(1+t)^{-\frac{1}{2}\left(1-\frac{1}{q}\right)}, \ \ t\ge0, \ 2\le q \le \infty.
\end{equation}
\end{prop}
%%%%%%%%%%%%%%%%%%%%%%%%%%%%%%%%%%%%%%%%%%%%%%%%%%%%
%%%%%%%%%%%%%%%%%%%%%%%%%%%%%%%%%%%%%%%%%%%%%%%%%%%%
\begin{proof}
We solve the integral equation \eqref{integral-eq} by using the contraction mapping principle for the mapping 
\begin{equation}\label{nonlinear-map}
N[u]:=T(t)*u_{0}-\int_{0}^{t}\p_{x}T(t-\tau)*(|u|^{p-1}u)(\tau)d\tau. 
\end{equation}
Let us introduce the Banach space $X$ as follows:
\begin{equation}\label{space-X}
X:=\left\{u\in C([0, \infty); H^{1}(\R)); \ \left\|u\right\|_{X}:=\sup_{t\ge0}\,(1+t)^{\frac{1}{4}}\left\|u(\cdot, t)\right\|_{L^{2}}+\sup_{t\ge0}\,(1+t)^{\frac{3}{4}}\left\|u_{x}(\cdot, t)\right\|_{L^{2}}<\infty \right\}.
\end{equation}
Now, we set $N_{0}:=T(t)*u_{0}$. Then, it follows from Lemma \ref{lem.linear-L2} that 
\begin{equation}\label{N0-est}
\exists C_{0}>0 \ \ \text{s.t.} \ \ \left\|N_{0}\right\|_{X}\le C_{0}E_{0}.
\end{equation}
In what follows, we apply the contraction mapping principle to \eqref{nonlinear-map} on the closed subset $Y$ of $X$ below: 
\begin{equation*}
Y:=\left\{u\in X; \ \left\|u\right\|_{X}\le2C_{0}E_{0}\right\}.
\end{equation*}
In order to complete the proof, it is sufficient to show the following estimates: 
\begin{equation}\label{contraction-1}
\left\|N[u]\right\|_{X}\le2C_{0}E_{0}, \ \ u \in Y,
\end{equation}
\begin{equation}\label{contraction-2}
\left\|N[u]-N[v]\right\|_{X}\le \frac{1}{2}\left\|u-v\right\|_{X}, \ \ u, v \in Y. 
\end{equation}
If we have shown \eqref{contraction-1} and \eqref{contraction-2}, 
by using the Banach fixed point theorem, we can see that \eqref{VFW} has a unique global mild solution in $Y$ satisfying the $L^{2}$-decay estimate \eqref{u-decay}. 

In the following, $E_{0}$ is assumed to be sufficiently small. First, from the Sobolev inequality 
\begin{equation}\label{Sobolev-ineq}
\left\|f\right\|_{L^{\infty}}\le \sqrt{2} \left\|f\right\|^{\frac{1}{2}}_{L^{2}}\left\|f'\right\|^{\frac{1}{2}}_{L^{2}}, \ \ f \in H^{1}(\R), 
\end{equation}
we have 
\begin{equation}\label{Linf-decay-pre}
\left\|u(\cdot, t)\right\|_{L^{\infty}}\le \left\|u\right\|_{X}(1+t)^{-\frac{1}{2}}. 
\end{equation}
In addition to \eqref{Linf-decay-pre}, we need to prepare the following estimates:
\begin{align}
\left\|\left(|u|^{p-1}u-|v|^{p-1}v\right)(\cdot, t)\right\|_{L^{1}}&\le C\left(\left\|u\right\|_{X}+\left\|v\right\|_{X}\right)\left\|u-v\right\|_{X}(1+t)^{-\frac{p-1}{2}}, \ \ u, v \in Y,\label{differ-L1}\\
\left\|\left(|u|^{p-1}u-|v|^{p-1}v\right)(\cdot, t)\right\|_{L^{2}}&\le C\left(\left\|u\right\|_{X}+\left\|v\right\|_{X}\right)\left\|u-v\right\|_{X}(1+t)^{-\frac{2p-1}{4}}, \ \ u, v \in Y, \label{differ-L2}\\
\left\|\p_{x}\left(|u|^{p-1}u-|v|^{p-1}v\right)(\cdot, t)\right\|_{L^{1}}&\le C\left(\left\|u\right\|_{X}+\left\|v\right\|_{X}\right)\left\|u-v\right\|_{X}(1+t)^{-\frac{p}{2}}, \ \ u, v \in Y, \label{differ-L1x} \\
\left\|\p_{x}\left(|u|^{p-1}u-|v|^{p-1}v\right)(\cdot, t)\right\|_{L^{2}}&\le C\left(\left\|u\right\|_{X}+\left\|v\right\|_{X}\right)\left\|u-v\right\|_{X}(1+t)^{-\frac{2p+1}{4}}, \ \ u, v \in Y. \label{differ-L2x}
\end{align}
We shall prove only $L^{1}$-decay estimates \eqref{differ-L1} and \eqref{differ-L1x}, since we can prove \eqref{differ-L2} and \eqref{differ-L2x} in the same way. 
First, we recall the following basic inequality: 
\begin{equation}\label{difference}
\left| |u|^{p-1}u-|v|^{p-1}v \right| \le C\left(|u|^{p-1}+|v|^{p-1}\right)|u-v|. 
\end{equation}
For \eqref{differ-L1}, by using the above inequality \eqref{difference}, the Schwarz inequality, \eqref{space-X} and \eqref{Linf-decay-pre}, we have 
\begin{align*}
&\left\|\left(|u|^{p-1}u-|v|^{p-1}v\right)(\cdot, t)\right\|_{L^{1}} \\
&\le C\left\|\left(|u|^{p-1}|u-v|\right)(\cdot, t)\right\|_{L^{1}}+C\left\|\left(|v|^{p-1}|u-v|\right)(\cdot, t)\right\|_{L^{1}} \\
&\le C\left\|u(\cdot, t)\right\|_{L^{\infty}}^{p-2}\left\|u(\cdot, t)\right\|_{L^{2}}\left\|(u-v)(\cdot, t)\right\|_{L^{2}}+C\left\|v(\cdot, t)\right\|_{L^{\infty}}^{p-2}\left\|v(\cdot, t)\right\|_{L^{2}}\left\|(u-v)(\cdot, t)\right\|_{L^{2}} \\
&\le C\left(\left\|u\right\|_{X}+\left\|v\right\|_{X}\right)\left\|u-v\right\|_{X}(1+t)^{-\frac{p-2}{2}}(1+t)^{-\frac{1}{4}}(1+t)^{-\frac{1}{4}} \\
&\le C\left(\left\|u\right\|_{X}+\left\|v\right\|_{X}\right)\left\|u-v\right\|_{X}(1+t)^{-\frac{p-1}{2}}. 
\end{align*}
For \eqref{differ-L1x}, noticing $\p_{x}(|u|^{p-1}u)=p|u|^{p-1}u_{x}$ and using the mean value theorem, we analogously have 
\begin{align*}
&\left\|\p_{x}\left(|u|^{p-1}u-|v|^{p-1}v\right)(\cdot, t)\right\|_{L^{1}} \\
&=p\left\|\left\{|u|^{p-1}(u-v)_{x}\right\}(\cdot, t)+\left\{\left(|u|^{p-1}-|v|^{p-1}\right)v_{x}\right\}(\cdot, t)\right\|_{L^{1}} \\
&\le C\left\|u(\cdot, t)\right\|_{L^{\infty}}^{p-2}\left\|u(\cdot, t)\right\|_{L^{2}}\left\|(u-v)_{x}(\cdot, t)\right\|_{L^{2}}+C\left\|\left(|u|^{p-1}-|v|^{p-1}\right)(\cdot, t)\right\|_{L^{2}}\left\|v_{x}(\cdot, t)\right\|_{L^{2}} \\
&\le C\left\|u(\cdot, t)\right\|_{L^{\infty}}^{p-2}\left\|u(\cdot, t)\right\|_{L^{2}}\left\|(u-v)_{x}(\cdot, t)\right\|_{L^{2}}\\
&\ \ \ \ +C(\left\|u(\cdot, t)\right\|_{L^{\infty}}^{p-2}+\left\|v(\cdot, t)\right\|_{L^{\infty}}^{p-2})\left\|\left(u-v\right)(\cdot, t)\right\|_{L^{2}}\left\|v_{x}(\cdot, t)\right\|_{L^{2}} \\
&\le C\left(\left\|u\right\|_{X}+\left\|v\right\|_{X}\right)\left\|u-v\right\|_{X}(1+t)^{-\frac{p-2}{2}}(1+t)^{-\frac{1}{4}}(1+t)^{-\frac{3}{4}} \\
&\le C\left(\left\|u\right\|_{X}+\left\|v\right\|_{X}\right)\left\|u-v\right\|_{X}(1+t)^{-\frac{p}{2}}. 
\end{align*}

Now, we would like to prove \eqref{contraction-1} and \eqref{contraction-2}. Recalling \eqref{nonlinear-map}, we obtain 
\begin{equation}\label{Nu-Nv}
\left(N[u]-N[v]\right)(t)=-\int_{0}^{t}\p_{x}T(t-\tau)*\left(|u|^{p-1}u-|v|^{p-1}v\right)(\tau)d\tau=:I(x, t). 
\end{equation}
Using the Plancherel theorem and splitting the $L^{2}$-norm of $I(x, t)$ as follows: 
\begin{equation}\label{Nu-Nv-split}
\left\|\p_{x}^{l}I(\cdot, t)\right\|_{L^{2}} \le \left\|(i\xi)^{l}\hat{I}(\xi, t)\right\|_{L^{2}(|\xi|\le1)}+\left\|(i\xi)^{l}\hat{I}(\xi, t)\right\|_{L^{2}(|\xi|\ge1)}=:I_{1}(t)+I_{2}(t), \ \ l=0, 1. 
\end{equation}
Since 
\begin{equation*}
\int_{|\xi|\le1}|\xi|^{j}\exp\left(-2(t-\tau)\xi^{2}\right)d\xi \le C(1+t-\tau)^{-\frac{j}{2}-\frac{1}{2}}, \ \ j\ge0, 
\end{equation*}
it follows from \eqref{differ-L1} and \eqref{differ-L1x} that 
\begin{align}
I_{1}(t)
&\le \int_{0}^{t}\left\|(i\xi)^{l+1}\exp\left(-\mu(t-\tau)\xi^{2}-\frac{i2Bb(t-\tau)\xi}{b^{2}+\xi^{2}}\right)\mathcal{F}\left[|u|^{p-1}u-|v|^{p-1}v\right](\xi, \tau)\right\|_{L^{2}(|\xi|\le1)}d\tau \nonumber \\
&\le \int_{0}^{t/2}\sup_{|\xi|\le1}\left|\mathcal{F}\left[|u|^{p-1}u-|v|^{p-1}v\right](\xi, \tau)\right|\left(\int_{|\xi|\le1}\xi^{2(l+1)}\exp\left(-2\mu(t-\tau)\xi^{2}\right)d\xi \right)^{\frac{1}{2}}d\tau \nonumber \\
&\ \ \ +\int_{t/2}^{t}\sup_{|\xi|\le1}\left|(i\xi)^{l}\mathcal{F}\left[|u|^{p-1}u-|v|^{p-1}v\right](\xi, \tau)\right|\left(\int_{|\xi|\le1}\xi^{2}\exp\left(-2\mu(t-\tau)\xi^{2}\right)d\xi \right)^{\frac{1}{2}}d\tau \nonumber \\
&\le C\int_{0}^{t/2}(1+t-\tau)^{-\frac{3}{4}-\frac{l}{2}}\left\|\left(|u|^{p-1}u-|v|^{p-1}v\right)(\cdot, \tau)\right\|_{L^{1}}d\tau \nonumber \\
&\ \ \ +C\int_{t/2}^{t}(1+t-\tau)^{-\frac{3}{4}}\left\|\p_{x}^{l}\left(|u|^{p-1}u-|v|^{p-1}v\right)(\cdot, \tau)\right\|_{L^{1}}d\tau \nonumber \\
&\le C\left(\left\|u\right\|_{X}+\left\|v\right\|_{X}\right)\left\|u-v\right\|_{X} \nonumber \\
&\ \ \ \times \left(\int_{0}^{t/2}(1+t-\tau)^{-\frac{3}{4}-\frac{l}{2}}(1+\tau)^{-\frac{p-1}{2}}d\tau+\int_{t/2}^{t}(1+t-\tau)^{-\frac{3}{4}}(1+\tau)^{-\frac{p-1}{2}-\frac{l}{2}}d\tau\right) \nonumber \\
&\le C\left(\left\|u\right\|_{X}+\left\|v\right\|_{X}\right)\left\|u-v\right\|_{X}
\begin{cases}
(1+t)^{-\frac{2p-3}{4}-\frac{l}{2}}, &2\le p<3, \\
(1+t)^{-\frac{3}{4}-\frac{l}{2}}\log(2+t), &p=3, \\
(1+t)^{-\frac{3}{4}-\frac{l}{2}}, &p>3, 
\end{cases} \nonumber \\
&\le C\left(\left\|u\right\|_{X}+\left\|v\right\|_{X}\right)\left\|u-v\right\|_{X}(1+t)^{-\frac{1}{4}-\frac{l}{2}}, \ \ p\ge2, \label{Nu-Nv-I1-est}
\end{align}
for any $t\ge 0$ and $l=0, 1$. 
Next, for $|\xi|\ge1$, by using the Schwarz inequality, we have 
\begin{align*}
\left|(i\xi)^{l}\hat{I}(\xi, t)\right|
&=\left|(i\xi)^{l+1}\int_{0}^{t}\exp\left(-\mu(t-\tau)\xi^{2}-\frac{i2Bb(t-\tau)\xi}{b^{2}+\xi^{2}}\right)\mathcal{F}\left[|u|^{p-1}u-|v|^{p-1}v\right](\xi, \tau)d\tau\right| \\
&\le \int_{0}^{t}|\xi|\exp\left(-\mu(t-\tau)\xi^{2}\right)\left|(i\xi)^{l}\mathcal{F}\left[|u|^{p-1}u-|v|^{p-1}v\right](\xi, \tau)\right|d\tau \\
&\le \left(\int_{0}^{t}\xi^2\exp\left(-\mu(t-\tau)\xi^{2}\right)d\tau \right)^{\frac{1}{2}} \\
&\ \ \ \times \left(\int_{0}^{t}\exp\left(-\mu(t-\tau)\xi^{2}\right)\left|(i\xi)^{l}\mathcal{F}\left[|u|^{p-1}u-|v|^{p-1}v\right](\xi, \tau)\right|^{2}d\tau \right)^{\frac{1}{2}} \\
&\le C\left(\int_{0}^{t}\exp\left(-\mu(t-\tau)\xi^{2}\right)\left|(i\xi)^{l}\mathcal{F}\left[|u|^{p-1}u-|v|^{p-1}v\right](\xi, \tau)\right|^{2}d\tau \right)^{\frac{1}{2}}.
\end{align*}
Therefore, it follows from \eqref{differ-L2} and \eqref{differ-L2x} that 
\begin{align}
I_{2}(t)
&\le C\left(\int_{|\xi|\ge1}\int_{0}^{t}\exp\left(-\mu(t-\tau)\xi^{2}\right)\left|(i\xi)^{l}\mathcal{F}\left[|u|^{p-1}u-|v|^{p-1}v\right](\xi, \tau)\right|^{2}d\tau d\xi\right)^{\frac{1}{2}} \nonumber \\
&\le C\left(\int_{0}^{t}\exp\left(-\mu(t-\tau)\right)\int_{|\xi|\ge1}\left|(i\xi)^{l}\mathcal{F}\left[|u|^{p-1}u-|v|^{p-1}v\right](\xi, \tau)\right|^{2}d\xi d\tau \right)^{\frac{1}{2}} \nonumber\\
&\le C\left(\int_{0}^{t}\exp\left(-\mu(t-\tau)\right)\left\|\p_{x}^{l}\left(|u|^{p-1}u-|v|^{p-1}v\right)(\cdot, \tau)\right\|_{L^{2}}^{2}d\tau \right)^{\frac{1}{2}} \nonumber\\
&\le C\left(\left\|u\right\|_{X}+\left\|v\right\|_{X}\right)\left\|u-v\right\|_{X}\left(\int_{0}^{t}\exp\left(-\mu(t-\tau)\right)(1+\tau)^{-\frac{2p-1}{2}-l}d\tau \right)^{\frac{1}{2}} \nonumber\\
&\le C\left(\left\|u\right\|_{X}+\left\|v\right\|_{X}\right)\left\|u-v\right\|_{X}(1+t)^{-\frac{2p-1}{4}-\frac{l}{2}} \nonumber \\
&\le C\left(\left\|u\right\|_{X}+\left\|v\right\|_{X}\right)\left\|u-v\right\|_{X}(1+t)^{-\frac{3}{4}-\frac{l}{2}}, \ \ p\ge2,  \label{Nu-Nv-I2-est}
\end{align}
for any $t\ge 0$ and $l=0, 1$. 
Combining \eqref{Nu-Nv} through \eqref{Nu-Nv-I2-est}, we obtain 
\begin{equation*}
\left\|\p_{x}^{l}\left(N[u]-N[v]\right)(t)\right\|_{L^{2}}\le C\left(\left\|u\right\|_{X}+\left\|v\right\|_{X}\right)\left\|u-v\right\|_{X}(1+t)^{-\frac{1}{4}-\frac{l}{2}}, \ \ t\ge0, \ l=0, 1.
\end{equation*}
Hence, there exists a positive constant $C_{1}>0$ such that 
\begin{equation*}
\left\|N[u]-N[v]\right\|_{X}\le C_{1}\left(\left\|u\right\|_{X}+\left\|v\right\|_{X}\right)\left\|u-v\right\|_{X}\le 4C_{0}C_{1}E_{0}\left\|u-v\right\|_{X}, \ \ u, v \in Y.
\end{equation*}
Here, we choose $E_{0}$ which satisfies $4C_{0}C_{1}E_{0} \le1/2$, then we have \eqref{contraction-2}. 
Moreover, we can see that \eqref{contraction-1} holds from \eqref{contraction-2}. Actually, taking $v=0$ in \eqref{contraction-2}, it follows that 
\begin{equation*}
\left\|N[u]-N[0]\right\|_{X}\le C_{0}E_{0}, \ \ u\in Y.
\end{equation*}
Therefore, since $N[0]=N_{0}$, we obtain from \eqref{contraction-2} that 
\begin{equation*}
\left\|N[u]\right\|_{X}\le \left\|N_{0}\right\|_{X}+\left\|N[u]-N[0]\right\|_{X} \le 2C_{0}E_{0}, \ \ u \in Y.
\end{equation*}
Thus, we get \eqref{contraction-1}. This completes the proof of the global existence and of the $L^{2}$-decay estimate \eqref{u-decay}. 

Finally, we shall prove \eqref{u-decay-Lq}. From the Sobolev inequality \eqref{Sobolev-ineq}, we immediately obtain 
\begin{equation}\label{Linf-decay}
\left\|u(\cdot, t)\right\|_{L^{\infty}}\le E_{0}(1+t)^{-\frac{1}{2}}, \ \ t\ge0.
\end{equation}
The estimate \eqref{u-decay-Lq} for $2<q<\infty$ can be obtained by \eqref{u-decay}, \eqref{Linf-decay} and an interpolation inequality 
\begin{equation}\label{interpolation}
\left\|f\right\|_{L^{q}}\le \left\|f\right\|^{1-\frac{2}{q}}_{L^{\infty}}\left\|f\right\|^{\frac{2}{q}}_{L^{2}}, \ \ 2<q<\infty, 
\end{equation}
as follows: 
\begin{align}
\left\|u(\cdot, t)\right\|_{L^{q}}
&\le \left\|u(\cdot, t)\right\|_{L^{\infty}}^{1-\frac{2}{q}}\left\|u(\cdot, t)\right\|_{L^{2}}^{\frac{2}{q}} \nonumber \\
&\le CE_{0}(1+t)^{-\frac{1}{2}(1-\frac{2}{q})}(1+t)^{-\frac{1}{2q}}
\le CE_{0}(1+t)^{-\frac{1}{2}(1-\frac{1}{q})}, \ \ t\ge0. \label{u-Lq}
\end{align}
This completes the proof. 
\end{proof}
%%%%%%%%%%%%%%%%%%%%%%%%%%%%%%%%%%%%%%%%%%%%%%%%%%%%

%%%%%%%%%%%%%%%%%%%
\section{Auxiliary Lemmas and Propositions}
%%%%%%%%%%%%%%%%%%%

In this section, we prepare some auxiliary lemmas and propositions to prove the main result.  
First, we introduce the asymptotic formula for the integral kernel $T(x, t)$. Now, we remark that 
\begin{align}
&\int_{\R}Be^{-b|x-y|}u_{y}(y, t)dy
=\mathcal{F}^{-1}\left[\frac{i2Bb\xi}{b^{2}+\xi^{2}}\hat{u}(\xi)\right](x) \nonumber \\
&=2Bb(b^{2}-\p_{x}^{2})^{-1}\p_{x}u
=\frac{2B}{b}\p_{x}u+\frac{2B}{b}(b^{2}-\p_{x}^{2})^{-1}\p_{x}^{3}u.  \label{dispersion}
\end{align}
Therefore, the integral kernel $T(x, t)$ is defined by \eqref{T-linear} can be rewritten by 
\begin{equation*}
T(x, t)=\frac{1}{\sqrt{2\pi}}\mathcal{F}^{-1}\left[\exp\left(-\mu t\xi^{2}-\frac{i2Bt\xi}{b}+\frac{i2Bt\xi^{3}}{b(b^{2}+\xi^{2})}\right)\right](x).
\end{equation*}
By using the above expression, we can show the following estimates (for the proof, see Lemma~4.1 in \cite{FI21}): 
%%%%%%%%%%%%%%%%%%%%%%%%%%%%%%%%%%%%%%%%%%%%%%%%%%%%
\begin{lem}\label{lem.T-G-linear}
Let $l$ be a non-negative integer and $2\le q\le \infty$. Then, we have 
\begin{align}
&\left\|\p_{x}^{l}T(\cdot, t)\right\|_{L^{q}}\le Ct^{-\frac{1}{2}\left(1-\frac{1}{q}\right)-\frac{l}{2}}, \ \ t>0, \label{T-basic} \\
&\left\|\p_{x}^{l}\left(T(\cdot, t)-G_{0}(\cdot, t)\right)\right\|_{L^{q}}\le Ct^{-\frac{1}{2}\left(1-\frac{1}{q}\right)-\frac{1}{2}-\frac{l}{2}}, \ \ t>0, \label{T-G-linear}
\end{align}
where $T(x, t)$ and $G_{0}(x, t)$ are defined by \eqref{T-linear} and \eqref{asymp-func}, respectively. 
\end{lem}
%%%%%%%%%%%%%%%%%%%%%%%%%%%%%%%%%%%%%%%%%%%%%%%%%%%%
By virtue of the above lemma, we can prove the following approximation formula for the Duhamel term in the integral equation \eqref{integral-eq}. 
The following result plays an important role of the proof of Theorem~\ref{thm.main}. 
%%%%%%%%%%%%%%%%%%%%%%%%%%%%%%%%%%%%%%%%%%%%%%%%%%%%
\begin{prop}\label{prop.asymp-Duhamel-pre}
Let $p>2$. Assume that $u_{0} \in H^{1}(\R) \cap L^{1}(\R)$ and $E_{0}=\|u_{0}\|_{H^{1}}+\|u_{0}\|_{L^{1}}$ is sufficiently small. 
Then, the solution $u(x, t)$ to \eqref{VFW} satisfies  
\begin{equation}\label{asymp-nonlinear-pre}
\lim_{t\to \infty}t^{\frac{1}{2}\left(1-\frac{1}{q}\right)+\frac{1}{2}}\left\|\int_{0}^{t}\p_{x}\left(T-G_{0}\right)(t-\tau)*\left(|u|^{p-1}u\right)(\tau)d\tau\right\|_{L^{q}}=0,
\end{equation}
for any $2\le q \le \infty$, where $T(x, t)$ and $G_{0}(x, t)$ are defined by \eqref{T-linear} and \eqref{asymp-func}, respectively. 
\end{prop}
%%%%%%%%%%%%%%%%%%%%%%%%%%%%%%%%%%%%%%%%%%%%%%%%%%%%
%%%%%%%%%%%%%%%%%%%%%%%%%%%%%%%%%%%%%%%%%%%%%%%%%%%%
\begin{proof}
First, we split the integral as follows: 
\begin{align}
&\int_{0}^{t}\p_{x}\left(T-G_{0}\right)(t-\tau)*\left(|u|^{p-1}u\right)(\tau)d\tau \nonumber \\
&=\left(\int_{0}^{t/2}+\int_{t/2}^{t}\right)\p_{x}\left(T-G_{0}\right)(t-\tau)*\left(|u|^{p-1}u\right)(\tau)d\tau=:D_{1}(x, t)+D_{2}(x, t). \label{Duhamel-split}
\end{align}

In order to evaluate $D_{1}(x, t)$ and $D_{2}(x, t)$, we need to prepare some decay estimates for $(|u|^{p-1}u)(x, t)$. 
It follows from the decay estimate \eqref{u-decay-Lq} that 
\begin{align}
\left\|\left(|u|^{p-1}u\right)(\cdot, t)\right\|_{L^{q}}
&\le \left\|u(\cdot, t)\right\|_{L^{\infty}}^{p-2}\left\|u(\cdot, t)\right\|_{L^{2q}}^{2} 
\le CE_{0}(1+t)^{-\frac{p-2}{2}}(1+t)^{-\left(1-\frac{1}{2q}\right)} \nonumber \\
&\le CE_{0}(1+t)^{-\frac{1}{2}\left(1-\frac{1}{q}\right)-\frac{p-1}{2}}, \ \ t\ge0, \ 1\le q\le \infty. \label{u^p-decay-Lq}
\end{align}
Moreover, we can see that the following estimate holds: 
\begin{equation}\label{d_u^p-decay-Lr}
\left\|\p_{x}\left(|u|^{p-1}u\right)(\cdot, t)\right\|_{L^{r}}
\le CE_{0}(1+t)^{-\frac{1}{2}\left(1-\frac{1}{r}\right)-\frac{p}{2}}, \ \ t\ge0, \ 1\le r\le 2.
\end{equation}
Actually, we get from \eqref{u-decay-Lq} and \eqref{u-decay} that 
\begin{align}
\left\|\p_{x}\left(|u|^{p-1}u\right)(\cdot, t)\right\|_{L^{2}}
&=p\left\|\left(|u|^{p-1}u_{x}\right)(\cdot, t)\right\|_{L^{2}} 
\le p\left\|u(\cdot, t)\right\|_{L^{\infty}}^{p-1}\left\|u_{x}(\cdot, t)\right\|_{L^{2}}  \nonumber \\
&\le CE_{0}(1+t)^{-\frac{p-1}{2}}(1+t)^{-\frac{3}{4}}
\le CE_{0}(1+t)^{-\frac{2p+1}{4}}, \ \ t\ge0. \label{d_u^p-decay-L2}
\end{align}
Moreover, using the Schwarz inequality, similarly we have 
\begin{align}
\left\|\p_{x}\left(|u|^{p-1}u\right)(\cdot, t)\right\|_{L^{1}}
&=p\left\|\left(|u|^{p-1}u_{x}\right)(\cdot, t)\right\|_{L^{1}} 
\le p\left\|u(\cdot, t)\right\|_{L^{\infty}}^{p-2}\left\|u(\cdot, t)\right\|_{L^{2}}\left\|u_{x}(\cdot, t)\right\|_{L^{2}} \nonumber \\
&\le CE_{0}(1+t)^{-\frac{p-2}{2}}(1+t)^{-\frac{1}{4}}(1+t)^{-\frac{3}{4}}
\le CE_{0}(1+t)^{-\frac{p}{2}}, \ \ t\ge0. \label{d_u^p-decay-L1}
\end{align}
Therefore, we can prove \eqref{d_u^p-decay-Lr} from \eqref{d_u^p-decay-L2} and \eqref{d_u^p-decay-L1} through an interpolation inequality 
\[
\left\|f\right\|_{L^{r}}\le \left\|f\right\|^{-1+\frac{2}{r}}_{L^{1}}\left\|f\right\|^{2-\frac{2}{r}}_{L^{2}}, \ \ 1<r<2, 
\]
as follows: 
\begin{align*}
&\left\|\p_{x}\left(|u|^{p-1}u\right)(\cdot, t)\right\|_{L^{r}}
\le \left\|\p_{x}\left(|u|^{p-1}u\right)(\cdot, t)\right\|_{L^{1}}^{-1+\frac{2}{r}}\left\|\p_{x}\left(|u|^{p-1}u\right)(\cdot, t)\right\|_{L^{2}}^{2-\frac{2}{r}} \\
&\le CE_{0}(1+t)^{-\frac{p}{2}\left(\frac{2}{r}-1\right)}(1+t)^{-\frac{2p+1}{4}\left(2-\frac{2}{r}\right)}
\le CE_{0}(1+t)^{-\frac{1}{2}\left(1-\frac{1}{r}\right)-\frac{p}{2}}, \ \ t\ge0, \ 1\le r\le 2.
\end{align*}

Now, we shall evaluate $D_{1}(x, t)$ and $D_{2}(x, t)$ in \eqref{Duhamel-split}. First for $D_{1}(x, t)$, it follows from Young's inequality, \eqref{T-G-linear} and \eqref{u^p-decay-Lq} that 
\begin{align}
\left\|D_{1}(\cdot, t)\right\|_{L^{q}} 
&\le \int_{0}^{t/2}\left\|\p_{x}(T-G_{0})(\cdot, t-\tau)\right\|_{L^{q}}\left\|\left(|u|^{p-1}u\right)(\cdot, \tau)\right\|_{L^{1}}d\tau \nonumber \\
&\le CE_{0}\int_{0}^{t/2}(t-\tau)^{-\frac{1}{2}\left(1-\frac{1}{q}\right)-1}(1+\tau)^{-\frac{p-1}{2}}d\tau \nonumber \\
&\le CE_{0}\begin{cases}
t^{-\frac{1}{2}\left(1-\frac{1}{q}\right)-\frac{p-1}{2}}, &2<p<3, \\
t^{-\frac{1}{2}\left(1-\frac{1}{q}\right)-1}\log(2+t), &p=3, \\
t^{-\frac{1}{2}\left(1-\frac{1}{q}\right)-1}, &p>3,
\end{cases}\label{D1-est}
\end{align}
for any $t>0$ and $2\le q\le \infty$. On the other hand, from Young's inequality, \eqref{T-G-linear} and \eqref{d_u^p-decay-Lr}, we get 
\begin{align}
\left\|D_{2}(\cdot, t)\right\|_{L^{q}} 
&\le \int_{t/2}^{t}\left\|(T-G_{0})(\cdot, t-\tau)\right\|_{L^{2}}\left\|\p_{x}\left(|u|^{p-1}u\right)(\cdot, \tau)\right\|_{L^{r}}d\tau \ \ \left(\frac{1}{q}+1=\frac{1}{2}+\frac{1}{r}\right) \nonumber \\
&\le CE_{0}\int_{t/2}^{t}(t-\tau)^{-\frac{3}{4}}(1+\tau)^{-\frac{1}{2}\left(\frac{1}{2}-\frac{1}{q}\right)-\frac{p}{2}}d\tau \nonumber \\
&\le CE_{0}t^{-\frac{1}{2}\left(1-\frac{1}{q}\right)-\frac{p-1}{2}}, \ \ t>0, \ 2\le q\le \infty. \label{D2-est}
\end{align}
Finally, combining \eqref{Duhamel-split}, \eqref{D1-est} and \eqref{D2-est}, we arrive at 
\begin{align*}
&\limsup_{t\to \infty}t^{\frac{1}{2}\left(1-\frac{1}{q}\right)+\frac{1}{2}}\left\|\int_{0}^{t}\p_{x}\left(T-G_{0}\right)(t-\tau)*\left(|u|^{p-1}u\right)(\tau)d\tau\right\|_{L^{q}} \\
&\le \limsup_{t\to \infty}t^{\frac{1}{2}\left(1-\frac{1}{q}\right)+\frac{1}{2}}\left\|D_{1}(\cdot, t)\right\|_{L^{q}} + \limsup_{t\to \infty}t^{\frac{1}{2}\left(1-\frac{1}{q}\right)+\frac{1}{2}}\left\|D_{2}(\cdot, t)\right\|_{L^{q}}  \\
&\le CE_{0}\lim_{t \to \infty}t^{-\frac{p-2}{2}}+CE_{0}\lim_{t\to \infty}\begin{cases}
t^{-\frac{p-2}{2}}, &2<p<3, \\
t^{-\frac{1}{2}}\log(2+t), &p=3, \\
t^{-\frac{1}{2}}, &p>3
\end{cases}\\
&=0, \ \ 2\le q\le \infty. 
\end{align*}
It means that the asymptotic formula \eqref{asymp-nonlinear-pre} holds. This completes the proof. 
\end{proof}
%%%%%%%%%%%%%%%%%%%%%%%%%%%%%%%%%%%%%%%%%%%%%%%%%%%%

Next, let us introduce several properties of the modified heat kernel $G_{0}(x, t)$ defined by \eqref{asymp-func}. 
This function satisfies the following decay estimate \eqref{heat-decay}. The proof is the same as the one for the usual heat kernel $G(x, t)$ defined by \eqref{kernel} (for details, see e.g.~\cite{GG99}).  
%%%%%%%%%%%%%%%%%%%%%%%%%%%%%%%%%%%%%%%%%%%%%%%%%%%%
\begin{lem}\label{lem.heat-decay}
Let $k$ and $l$ be non-negative integers. Then, for $1\le q\le \infty$, we have
\begin{equation}\label{heat-decay}
\left\| \p_{t}^{k}\p_{x}^{l}G_{0}(\cdot, t) \right\|_{L^{q}}\le Ct^{-\frac{1}{2}\left(1-\frac{1}{q}\right)-\frac{l+k}{2}}, \ \ t>0.
\end{equation}
\end{lem}
%%%%%%%%%%%%%%%%%%%%%%%%%%%%%%%%%%%%%%%%%%%%%%%%%%%%
Moreover, we can prove the following asymptotic formula for $G_{0}(x, t)$: 
%%%%%%%%%%%%%%%%%%%%%%%%%%%%%%%%%%%%%%%%%%%%%%%%%%%%
\begin{prop}\label{prop.asymp-heat}
Let $l$ be a non-negative integer and $1\le q\le \infty$. 
Suppose $u_{0}\in L^{1}(\R)$ and $xu_{0}\in L^{1}(\R)$. Then, we have
\begin{align}
&\left\|\p_{x}^{l}\left(G_{0}(t)*u_{0}-MG_{0}(\cdot, t)\right)\right\|_{L^{q}} \le C\|xu_{0}\|_{L^{1}}t^{-\frac{1}{2}\left(1-\frac{1}{q}\right)-\frac{1}{2}-\frac{l}{2}}, \ \ t>0, \label{heat-first} \\
&\lim_{t\to \infty}t^{\frac{1}{2}\left(1-\frac{1}{q}\right)+\frac{1}{2}+\frac{l}{2}}\left\|\p_{x}^{l}\left(G_{0}(t)*u_{0}-MG_{0}(\cdot, t)+m\p_{x}G_{0}(\cdot, t)\right)\right\|_{L^{q}}=0, \label{heat-second}
\end{align}
where $G_{0}(x, t)$, $M$ and $m$ are defined by \eqref{asymp-func} and \eqref{M-m-notation}, respectively. 
\end{prop}
%%%%%%%%%%%%%%%%%%%%%%%%%%%%%%%%%%%%%%%%%%%%%%%%%%%%
%%%%%%%%%%%%%%%%%%%%%%%%%%%%%%%%%%%%%%%%%%%%%%%%%%%%
\begin{proof}
We shall only prove \eqref{heat-second} because \eqref{heat-first} can be derived by a standard way (see e.g.~\cite{GG99}). 
First, it follows from the definition of $M$ given by \eqref{M-m-notation} that 
\begin{equation}\label{Gu0-MG}
G_{0}(t)*u_{0}-MG_{0}(x, t)
=\int_{\R}\left(G_{0}(x-y, t)-G_{0}(x, t)\right)u_{0}(y)dy.
\end{equation}
Now, recalling Taylor's theorem, we have
\[
f(1)
=f(0)+\int_{0}^{1}f'(\theta)d\theta
=f(0)+f'(0)+\int_{0}^{1}f''(\theta)(1-\theta)d\theta. 
\]
Then, applying the above formula for $f(\theta):=G_{0}(x-\theta y, t)$ for $\theta \in \R$, we obtain 
\begin{align}
&G_{0}(x-y, t)=G_{0}(x, t)-y\left(\int_{0}^{1}\p_{x}G_{0}(x-\theta y, t)d\theta \right), \label{G-Taylor-1}\\
&
G_{0}(x-y, t)=G_{0}(x, t)-y\p_{x}G_{0}(x, t)+y^{2}\left(\int_{0}^{1} \p_{x}^{2}G_{0}(x-\theta y, t)(1-\theta) d\theta \right). \label{G-Taylor-2}
\end{align}
Since $xu_{0}\in L^{1}(\R)$, for any $\e_{0}>0$, there exists a constant $L=L(\e_{0})>0$ such that 
$\int_{|y| \ge L} \left|yu_{0}(y)\right| dy < \e_{0}$. Now, we split the $y$-integral in \eqref{Gu0-MG} as follows: 
\begin{align}
&G_{0}(t)*u_{0}-MG_{0}(x, t)+m\p_{x}G_{0}(x, t) \nonumber \\
&=\int_{\R}\left(G_{0}(x-y, t)-G_{0}(x, t)+y\p_{x}G_{0}(x, t)\right)u_{0}(y)dy \nonumber \\
&=\int_{|y|\le L} \left(\int_{0}^{1} \p_{x}^{2}G_{0}(x-\theta y, t)(1-\theta) d\theta \right)y^{2}u_{0}(y)dy \nonumber \\
&\ \ \ \ +\int_{|y|\ge L} \left\{\p_{x}G_{0}(x, t)-\left(\int_{0}^{1}\p_{x}G_{0}(x-\theta y, t)d\theta \right)\right\}yu_{0}(y)dy, \label{Gu0-MG+mpG}
\end{align}
where we have used the facts \eqref{G-Taylor-1} and \eqref{G-Taylor-2}. Then, from \eqref{Gu0-MG+mpG} and Lemma~\ref{lem.heat-decay}, we can see that 
\begin{align*}
&\left\|\p_{x}^{l}\left(G_{0}(t)*u_{0}-MG_{0}(\cdot, t)+m\p_{x}G_{0}(\cdot, t)\right)\right\|_{L^{q}} \\
&\le \int_{|y|\le L} \left(\int_{0}^{1} \left\|\p_{x}^{l+2}G_{0}(\cdot-\theta y, t)(1-\theta) \right\|_{L^{q}}d\theta \right)|y|^{2}|u_{0}(y)|dy \\
&\ \ \ \ +\int_{|y|\ge L} \left\{\left\|\p_{x}^{l+1}G_{0}(\cdot, t) \right\|_{L^{q}}+\left(\int_{0}^{1}\left\|\p_{x}^{l+1}G_{0}(\cdot-\theta y, t)\right\|_{L^{q}}d\theta \right)\right\}|yu_{0}(y)|dy \\
&\le CL^{2}\|u_{0}\|_{L^{1}}t^{-\frac{1}{2}\left(1-\frac{1}{q}\right)-\frac{l+2}{2}}
+\e_{0}Ct^{-\frac{1}{2}\left(1-\frac{1}{q}\right)-\frac{l+1}{2}}, \ \ t>0.
\end{align*}
Thus, we finally arrive at
\[
\limsup_{t\to \infty}t^{\frac{1}{2}\left(1-\frac{1}{q}\right)+\frac{1}{2}+\frac{l}{2}}\left\|\p_{x}^{l}\left(G_{0}(t)*u_{0}-MG_{0}(\cdot, t)+m\p_{x}G_{0}(\cdot, t)\right)\right\|_{L^{q}}\le \e_{0}C. 
\]
Therefore, we get \eqref{heat-second}, because $\e_{0}>0$ can be chosen arbitrarily small.
\end{proof}
%%%%%%%%%%%%%%%%%%%%%%%%%%%%%%%%%%%%%%%%%%%%%%%%%%%%

Finally in this section, we would like to introduce some useful lemmas to prove the main theorem. 
To doing that, we first show that the solution $u(x, t)$ to \eqref{VFW} can be approximated by $MG_{0}(x, t)$. 
%%%%%%%%%%%%%%%%%%%%%%%%%%%%%%%%%%%%%%%%%%%%%%%%%%%%
\begin{prop}\label{prop.asymp-nonlinear}
Let $p>2$. Assume that $u_{0} \in H^{1}(\R) \cap L^{1}(\R)$, $xu_{0} \in L^{1}(\R)$ and $E_{0}=\|u_{0}\|_{H^{1}}+\|u_{0}\|_{L^{1}}$ is sufficiently small. 
Then, the solution $u(x, t)$ to \eqref{VFW} satisfies 
\begin{align}
\left\|u(\cdot, t)-MG_{0}(\cdot, t)\right\|_{L^{q}} 
\le CE_{1}
\begin{cases}
t^{-\frac{1}{2}\left(1-\frac{1}{q}\right)-\frac{p-2}{2}}, &2<p<3, \\
t^{-\frac{1}{2}\left(1-\frac{1}{q}\right)-\frac{1}{2}}\log(2+t), &p=3, \\
t^{-\frac{1}{2}\left(1-\frac{1}{q}\right)-\frac{1}{2}}, &p>3,
\end{cases}\label{asymp-nonlinear-0}
\end{align}
for any $t\ge1$ and $2\le q\le \infty$, where $G_{0}(x, t)$ and $M$ are defined by \eqref{asymp-func} and \eqref{M-m-notation}, respectively. 
Also, the above constant $E_{1}$ is defined by $E_{1}:=E_{0}+\|xu_{0}\|_{L^{1}}$. 
\end{prop}
%%%%%%%%%%%%%%%%%%%%%%%%%%%%%%%%%%%%%%%%%%%%%%%%%%%%
%%%%%%%%%%%%%%%%%%%%%%%%%%%%%%%%%%%%%%%%%%%%%%%%%%%%
\begin{proof}
We rewrite the integral equation \eqref{integral-eq} as follows: 
\begin{align}
u(x, t)-MG_{0}(x, t)
&=\left(T-G_{0}\right)(t)*u_{0}+G_{0}(t)*u_{0}-MG_{0}(x, t) \nonumber \\
&\ \ \ \ -\left(\int_{0}^{t/2}+\int_{t/2}^{t}\right)\p_{x}T(t-\tau)*(|u|^{p-1}u)(\tau)d\tau \nonumber \\
&=:\left(T-G_{0}\right)(t)*u_{0}+\left\{G_{0}(t)*u_{0}-MG_{0}(x, t)\right\}+N_{1}(x, t)+N_{2}(x, t). \label{IE-first}
\end{align}
Then, from Lemma~\ref{lem.T-G-linear} and Young's inequality, we can easily show 
\begin{align}
\left\|\left(T-G_{0}\right)(t)*u_{0}\right\|_{L^{q}}
&\le \left\|T(\cdot, t)-G_{0}(\cdot, t)\right\|_{L^{q}}\left\|u_{0}\right\|_{L^{1}} \nonumber \\
&\le C\left\|u_{0}\right\|_{L^{1}}t^{-\frac{1}{2}\left(1-\frac{1}{q}\right)-\frac{1}{2}}, \ \ t>0, \ 2\le q\le \infty. \label{IE-1stT}
\end{align}
For the second term in the right hand side of \eqref{IE-first}, we just recall \eqref{heat-first}: 
\begin{equation}\tag{\ref{heat-first}}
\left\|G_{0}(t)*u_{0}-MG_{0}(\cdot, t)\right\|_{L^{q}} \le C\|xu_{0}\|_{L^{1}}t^{-\frac{1}{2}\left(1-\frac{1}{q}\right)-\frac{1}{2}}, \ \ t>0, \ 1\le q\le \infty. \label{IE-2ndT}
\end{equation}

In order to prove \eqref{asymp-nonlinear-0}, we only need to evaluate the Duhamel term of \eqref{IE-first}. 
First, let us evaluate $N_{1}(x, t)$. Similarly as \eqref{D1-est}, it follows from Young's inequality, \eqref{T-basic} and \eqref{u^p-decay-Lq} that 
\begin{align}
\left\|N_{1}(\cdot, t)\right\|_{L^{q}} 
&\le \int_{0}^{t/2}\left\|\p_{x}T(\cdot, t-\tau)\right\|_{L^{q}}\left\|\left(|u|^{p-1}u\right)(\cdot, \tau)\right\|_{L^{1}}d\tau \nonumber \\
&\le CE_{0}\int_{0}^{t/2}(t-\tau)^{-\frac{1}{2}\left(1-\frac{1}{q}\right)-\frac{1}{2}}(1+\tau)^{-\frac{p-1}{2}}d\tau \nonumber \\
&\le CE_{0}\begin{cases}
t^{-\frac{1}{2}\left(1-\frac{1}{q}\right)-\frac{p-2}{2}}, &2<p<3, \\
t^{-\frac{1}{2}\left(1-\frac{1}{q}\right)-\frac{1}{2}}\log(2+t), &p=3, \\
t^{-\frac{1}{2}\left(1-\frac{1}{q}\right)-\frac{1}{2}}, &p>3,
\end{cases}\label{N1-est}
\end{align}
for any $t>0$ and $2\le q \le \infty$. On the other hand, for $N_{2}(x, t)$, in the same way to get \eqref{D2-est}, we have from Young's inequality, \eqref{T-basic} and \eqref{d_u^p-decay-Lr} that 
\begin{align}
\left\|N_{2}(\cdot, t)\right\|_{L^{q}} 
&\le \int_{t/2}^{t}\left\|T(\cdot, t-\tau)\right\|_{L^{2}}\left\|\p_{x}\left(|u|^{p-1}u\right)(\cdot, \tau)\right\|_{L^{r}}d\tau \ \ \left(\frac{1}{q}+1=\frac{1}{2}+\frac{1}{r}\right) \nonumber \\
&\le CE_{0}\int_{t/2}^{t}(t-\tau)^{-\frac{1}{4}}(1+\tau)^{-\frac{1}{2}\left(\frac{1}{2}-\frac{1}{q}\right)-\frac{p}{2}}d\tau \nonumber \\
&\le CE_{0}t^{-\frac{1}{2}\left(1-\frac{1}{q}\right)-\frac{p-2}{2}}, \ \ t>0, \ 2\le q\le \infty. \label{N2-est}
\end{align}
Combining \eqref{IE-first} through \eqref{N2-est}, we arrive at \eqref{asymp-nonlinear-0}. This completes the proof. 
\end{proof}
%%%%%%%%%%%%%%%%%%%%%%%%%%%%%%%%%%%%%%%%%%%%%%%%%%%%
By virtue of Proposition~\ref{prop.asymp-nonlinear}, we can get the following two lemmas,
 which will be used in the proofs of \eqref{u-asymp-2<p<3} and \eqref{u-asymp-p=3}. 
 The methods of the proofs of these lemmas are based on the techniques used for the generalized KdV--Burgers equation and the generalized BBM--Burgers equation (see Lemma~5.3 in \cite{K99}).
%%%%%%%%%%%%%%%%%%%%%%%%%%%%%%%%%%%%%%%%%%%%%%%%%%%%
\begin{lem}\label{lem.asymp-nonlinear-1}
Let $p>2$. Assume that $u_{0} \in H^{1}(\R) \cap L^{1}(\R)$, $xu_{0} \in L^{1}(\R)$ and $E_{0}=\|u_{0}\|_{H^{1}}+\|u_{0}\|_{L^{1}}$ is sufficiently small. 
Then, the solution $u(x, t)$ to \eqref{VFW} satisfies 
\begin{align}
\left\|\left(|u|^{p-1}u\right)(\cdot, t)-\left(|MG_{0}|^{p-1}MG_{0}\right)(\cdot, t)\right\|_{L^{q}} 
\le CE_{1}
\begin{cases}
t^{-\frac{1}{2}\left(1-\frac{1}{q}\right)-\frac{2p-3}{2}}, &2<p<3, \\
t^{-\frac{1}{2}\left(1-\frac{1}{q}\right)-\frac{3}{2}}\log(2+t), &p=3, \\
t^{-\frac{1}{2}\left(1-\frac{1}{q}\right)-\frac{3}{2}}, &p>3,
\end{cases}\label{asymp-nonlinear-1}
\end{align}
for any $t\ge1$ and $2\le q\le \infty$, where $G_{0}(x, t)$ and $M$ are defined by \eqref{asymp-func} and \eqref{M-m-notation}, respectively. 
Also, the constant $E_{1}$ is defined by $E_{1}=E_{0}+\|xu_{0}\|_{L^{1}}$. 
\end{lem}
%%%%%%%%%%%%%%%%%%%%%%%%%%%%%%%%%%%%%%%%%%%%%%%%%%%%
%%%%%%%%%%%%%%%%%%%%%%%%%%%%%%%%%%%%%%%%%%%%%%%%%%%%
\begin{proof}
By using the inequality \eqref{difference}, \eqref{u-decay-Lq}, Lemma~\ref{lem.heat-decay} and Proposition~\ref{prop.asymp-nonlinear}, we obtain 
\begin{align*}
&\left\|\left(|u|^{p-1}u\right)(\cdot, t)-\left(|MG_{0}|^{p-1}MG_{0}\right)(\cdot, t)\right\|_{L^{q}}  \\
&\le C\left(\left\|u(\cdot, t)\right\|_{L^{\infty}}^{p-1}+|M|^{p-1}\left\|G_{0}(\cdot, t)\right\|_{L^{\infty}}^{p-1}\right)\left\|u(\cdot, t)-MG_{0}(\cdot, t)\right\|_{L^{q}} \\
&\le C\left\{E^{p-1}(1+t)^{-\frac{p-1}{2}}+|M|^{p-1}t^{-\frac{p-1}{2}}\right\}
\cdot CE_{1}
\begin{cases}
t^{-\frac{1}{2}\left(1-\frac{1}{q}\right)-\frac{p-2}{2}}, &2<p<3, \\
t^{-\frac{1}{2}\left(1-\frac{1}{q}\right)-\frac{1}{2}}\log(2+t), &p=3, \\
t^{-\frac{1}{2}\left(1-\frac{1}{q}\right)-\frac{1}{2}}, &p>3
\end{cases} \\
&\le CE_{1}
\begin{cases}
t^{-\frac{1}{2}\left(1-\frac{1}{q}\right)-\frac{2p-3}{2}}, &2<p<3, \\
t^{-\frac{1}{2}\left(1-\frac{1}{q}\right)-\frac{3}{2}}\log(2+t), &p=3, \\
t^{-\frac{1}{2}\left(1-\frac{1}{q}\right)-\frac{3}{2}}, &p>3,
\end{cases}
\end{align*}
for any $t\ge1$ and $2\le q\le \infty$. This completes the proof. 
\end{proof}
%%%%%%%%%%%%%%%%%%%%%%%%%%%%%%%%%%%%%%%%%%%%%%%%%%%%

%%%%%%%%%%%%%%%%%%%%%%%%%%%%%%%%%%%%%%%%%%%%%%%%%%%%
\begin{lem}\label{lem.asymp-nonlinear-2}
Let $p>2$. Assume that $u_{0} \in H^{1}(\R) \cap L^{1}(\R)$, $xu_{0} \in L^{1}(\R)$ and $E_{0}=\|u_{0}\|_{H^{1}}+\|u_{0}\|_{L^{1}}$ is sufficiently small. 
Then, there exists a positive function $\eta \in L^{\infty}(0, \infty)$ satisfying $\lim_{t\to \infty}\eta(t)=0$ such that the solution $u(x, t)$ to \eqref{VFW} satisfies 
\begin{align}
t^{\frac{p-1}{2}}\left\|\left(|u|^{p-1}u\right)(\cdot, t)-\left(|MG_{0}|^{p-1}MG_{0}\right)(\cdot, t)\right\|_{L^{1}} 
\le \eta(t), \ \ t>0, \label{asymp-nonlinear-2}
\end{align}
where $G_{0}(x, t)$ and $M$ are defined by \eqref{asymp-func} and \eqref{M-m-notation}, respectively. 
\end{lem}
%%%%%%%%%%%%%%%%%%%%%%%%%%%%%%%%%%%%%%%%%%%%%%%%%%%%
%%%%%%%%%%%%%%%%%%%%%%%%%%%%%%%%%%%%%%%%%%%%%%%%%%%%
\begin{proof}
It follows from the inequality \eqref{difference}, \eqref{u-decay-Lq}, Lemma~\ref{lem.heat-decay} and Proposition~\ref{prop.asymp-nonlinear} that 
\begin{align}
&t^{\frac{p-1}{2}}\left\|\left(|u|^{p-1}u\right)(\cdot, t)-\left(|MG_{0}|^{p-1}MG_{0}\right)(\cdot, t)\right\|_{L^{1}}  \nonumber \\
&\le Ct^{\frac{p-1}{2}}\left(\left\||u|^{p-1}(\cdot, t)\right\|_{L^{2}}+|M|^{p-1}\left\|G_{0}^{p-1}(\cdot, t)\right\|_{L^{2}}\right)\left\|u(\cdot, t)-MG_{0}(\cdot, t)\right\|_{L^{2}} \nonumber\\
&\le Ct^{\frac{p-1}{2}}\left(\left\|u(\cdot, t)\right\|_{L^{\infty}}^{p-2}\left\|u(\cdot, t)\right\|_{L^{2}}+|M|^{p-1}\left\|G_{0}(\cdot, t)\right\|_{L^{\infty}}^{p-2}\left\|G_{0}(\cdot, t)\right\|_{L^{2}}\right)\left\|u(\cdot, t)-MG_{0}(\cdot, t)\right\|_{L^{2}} \nonumber\\
&\le Ct^{\frac{p-1}{2}}\left\{E^{p-1}(1+t)^{-\frac{p-2}{2}}(1+t)^{-\frac{1}{4}}+|M|^{p-1}t^{-\frac{p-2}{2}}t^{-\frac{1}{4}}\right\}
\cdot CE_{1}
\begin{cases}
t^{-\frac{1}{4}-\frac{p-2}{2}}, &2<p<3, \\
t^{-\frac{3}{4}}\log(2+t), &p=3, \\
t^{-\frac{3}{4}}, &p>3
\end{cases} \nonumber\\
&\le CE_{1}
\begin{cases}
t^{-\frac{p-2}{2}}, &2<p<3, \\
t^{-\frac{1}{2}}\log(2+t), &p=3, \\
t^{-\frac{1}{2}}, &p>3
\end{cases}\nonumber \\
&=:g(t), \ \ t\ge1. \label{decay-part}
\end{align}
Moreover, we have from \eqref{u^p-decay-Lq} and Lemma~\ref{lem.heat-decay} that 
\begin{align}
&t^{\frac{p-1}{2}}\left\|\left(|u|^{p-1}u\right)(\cdot, t)-\left(|MG_{0}|^{p-1}MG_{0}\right)(\cdot, t)\right\|_{L^{1}} \nonumber\\
&\le t^{\frac{p-1}{2}}\left\{\left\|\left(|u|^{p-1}u\right)(\cdot, t)\right\|_{L^{1}}+t^{\frac{p-1}{2}}\left\|\left(|MG_{0}|^{p-1}MG_{0}\right)(\cdot, t)\right\|_{L^{1}}\right\} \nonumber\\
&\le t^{\frac{p-1}{2}}\left\{CE_{0}(1+t)^{-\frac{p-1}{2}}+|M|^{p}\left\|G_{0}(\cdot, t)\right\|_{L^{\infty}}^{p-1}\left\|G_{0}(\cdot, t)\right\|_{L^{1}}\right\} \nonumber\\
&\le t^{\frac{p-1}{2}}\left\{CE_{0}t^{-\frac{p-1}{2}}+C|M|^{p}t^{-\frac{p-1}{2}}\right\}
\le CE_{0}=:C_{0}, \ \ t>0. \label{bdd-part}
\end{align}
Therefore, combining \eqref{decay-part} and \eqref{bdd-part}, we can conclude that \eqref{asymp-nonlinear-2} is true with the following $\eta(t)$:
\[
\eta(t):=
\begin{cases}
g(t), &t\ge1, \\
C_{0}, &0<t<1.
\end{cases}
\]
This completes the proof. 
\end{proof}
%%%%%%%%%%%%%%%%%%%%%%%%%%%%%%%%%%%%%%%%%%%%%%%%%%%%

%%%%%%%%%%%%%%%%%%%
\section{Proof of the Main Result}
%%%%%%%%%%%%%%%%%%%

In this section, we shall prove our main result Theorem~\ref{thm.main}, i.e. we establish the asymptotic formulas \eqref{u-asymp-2<p<3}, \eqref{u-asymp-p=3} and \eqref{u-asymp-p>3}. This section is divided into three subsections below. 

%%%%%%%%%%%%%%%%%%%
\subsection{Proof of Theorem~\ref{thm.main} for $\bm{2<p<3}$}
%%%%%%%%%%%%%%%%%%%

First in this subsection, we would like to prove Theorem~\ref{thm.main} in the case of $2<p<3$, i.e. we shall prove \eqref{u-asymp-2<p<3}. 
To doing that, we prepare the following approximation formula for the Duhamel term of \eqref{integral-eq}. 
%%%%%%%%%%%%%%%%%%%%%%%%%%%%%%%%%%%%%%%%%%%%%%%%%%%%
\begin{prop}
\label{prop.nl-2<p<3}
Let $2<p<3$. Assume that $u_{0} \in H^{1}(\R) \cap L^{1}(\R)$, $xu_{0}\in L^{1}(\R)$ and $E_{0}=\|u_{0}\|_{H^{1}}+\|u_{0}\|_{L^{1}}$ is sufficiently small. 
Then, the solution $u(x, t)$ to \eqref{VFW} satisfies 
\begin{equation}\label{nl-2<p<3}
\lim_{t\to \infty}t^{\frac{1}{2}\left(1-\frac{1}{q}\right)+\frac{p-2}{2}}\left\|\int_{0}^{t}\p_{x}T(t-\tau)*\left(|u|^{p-1}u\right)(\tau)d\tau-\left(|M|^{p-1}M\right)W_{p}(\cdot, t)\right\|_{L^{q}}=0, 
\end{equation}
for any $2\le q \le \infty$, where $T(x, t)$, $W_{p}(x, t)$ and $M$ are defined by \eqref{T-linear}, \eqref{asymp-func} and \eqref{M-m-notation}, respectively. 
\end{prop}
%%%%%%%%%%%%%%%%%%%%%%%%%%%%%%%%%%%%%%%%%%%%%%%%%%%%
%%%%%%%%%%%%%%%%%%%%%%%%%%%%%%%%%%%%%%%%%%%%%%%%%%%%
\begin{proof}
First of all, for simplicity, we set 
\begin{equation}\label{K-DEF}
K(x, t):=\int_{0}^{t}\p_{x}G_{0}(t-\tau)*\left(|u|^{p-1}u-|MG_{0}|^{p-1}MG_{0}\right)(\tau)d\tau. 
\end{equation}
In what follows, we shall prove 
\begin{equation}\label{K-limit}
\lim_{t\to \infty}t^{\frac{1}{2}\left(1-\frac{1}{q}\right)+\frac{p-2}{2}}\left\|K(\cdot, t)\right\|_{L^{q}}=0, \ \ 2\le q\le \infty. 
\end{equation}
We start with the evaluation for the $L^{q}$-norm of $K(x, t)$ in the case of $2\le q<\infty$. 
From Young's inequality, Lemmas~\ref{lem.heat-decay} and \ref{lem.asymp-nonlinear-2} and the change of valuable, we obtain 
\begin{align*}
\left\|K(\cdot, t)\right\|_{L^{q}} 
&\le \int_{0}^{t}\left\|\p_{x}G_{0}(\cdot, t-\tau)\right\|_{L^{q}}\left\|\left(|u|^{p-1}u\right)(\cdot, \tau)-\left(|MG_{0}|^{p-1}MG_{0}\right)(\cdot, \tau)\right\|_{L^{1}}d\tau \nonumber \\
&\le C\int_{0}^{t}(t-\tau)^{-\frac{1}{2}\left(1-\frac{1}{q}\right)-\frac{1}{2}}\tau^{-\frac{p-1}{2}}\eta(\tau)d\tau \\
&=Ct^{-\frac{1}{2}\left(1-\frac{1}{q}\right)-\frac{p-2}{2}}\int_{0}^{1}(1-s)^{-\frac{1}{2}\left(1-\frac{1}{q}\right)-\frac{1}{2}}s^{-\frac{p-1}{2}}\eta(st)ds \quad (\tau=st).
\end{align*}
Therefore, from Lebesgue's dominated convergence theorem, we get 
\begin{align}
&\limsup_{t\to \infty}t^{\frac{1}{2}\left(1-\frac{1}{q}\right)+\frac{p-2}{2}}\left\|K(\cdot, t)\right\|_{L^{q}}  \nonumber \\
&\le C\lim_{t\to \infty}\int_{0}^{1}(1-s)^{-\frac{1}{2}\left(1-\frac{1}{q}\right)-\frac{1}{2}}s^{-\frac{p-1}{2}}\eta(st)ds=0, \ \ 2\le q<\infty.  \label{K-limit-Lq}
\end{align}
To handle the $L^{\infty}$-norm, splitting the $\tau$-integral and using Young's inequality, Lemmas~\ref{lem.heat-decay}, \ref{lem.asymp-nonlinear-2} and \ref{lem.asymp-nonlinear-1} and the change of valuable, we have 
\begin{align*}
\left\|K(\cdot, t)\right\|_{L^{\infty}} 
&\le \int_{0}^{t/2}\left\|\p_{x}G_{0}(\cdot, t-\tau)\right\|_{L^{\infty}}\left\|\left(|u|^{p-1}u\right)(\cdot, \tau)-\left(|MG_{0}|^{p-1}MG_{0}\right)(\cdot, \tau)\right\|_{L^{1}}d\tau \nonumber \\
&\ \ \ +\int_{t/2}^{t}\left\|\p_{x}G_{0}(\cdot, t-\tau)\right\|_{L^{1}}\left\|\left(|u|^{p-1}u\right)(\cdot, \tau)-\left(|MG_{0}|^{p-1}MG_{0}\right)(\cdot, \tau)\right\|_{L^{\infty}}d\tau \nonumber \\
&\le C\int_{0}^{t/2}(t-\tau)^{-1}\tau^{-\frac{p-1}{2}}\eta(\tau)d\tau 
+CE_{1}\int_{t/2}^{t}(t-\tau)^{-\frac{1}{2}}\tau^{-(p-1)}d\tau \\
&= Ct^{-\frac{p-1}{2}}\int_{0}^{1/2}(1-s)^{-1}s^{-\frac{p-1}{2}}\eta(st)ds
+CE_{1}\int_{t/2}^{t}(t-\tau)^{-\frac{1}{2}}\tau^{-(p-1)}d\tau \\
&\le Ct^{-\frac{p-1}{2}}\int_{0}^{1/2}s^{-\frac{p-1}{2}}\eta(st)ds
+CE_{1}t^{-p+\frac{3}{2}}, \ \ t\ge2.
\end{align*}
Thus, we can get the following result similarly as \eqref{K-limit-Lq}: 
\begin{align}
\limsup_{t\to \infty}t^{\frac{p-1}{2}}\left\|K(\cdot, t)\right\|_{L^{\infty}}
\le C\lim_{t\to \infty}\int_{0}^{1/2}s^{-\frac{p-1}{2}}\eta(st)ds+CE_{1}\lim_{t\to \infty}t^{-\frac{p-2}{2}}=0.  \label{K-limit-Linf}
\end{align}
Combining \eqref{K-limit-Lq} and \eqref{K-limit-Linf}, we can say that \eqref{K-limit} is true. 

Next, let us derive $W_{p}(x, t)$ defined by \eqref{asymp-func}. To simplify the calculation, we set 
\begin{equation}\label{alpha-DEF}
\alpha:=\frac{2B}{b}. 
\end{equation}
Recalling the definitions of $G_{0}(x, t)$ and $G(x, t)$ (i.e. \eqref{asymp-func} and \eqref{kernel}, respectively) and using the change of valuable several times, 
we are able to see that  
\begin{align}
\int_{0}^{t}G_{0}(t-\tau)*G_{0}^{p}(\tau)d\tau
&=\int_{0}^{t}\int_{\R}G_{0}(x-y, t-\tau)G_{0}^{p}(y, \tau)dyd\tau \nonumber\\
&=\int_{0}^{t}\int_{\R}G\left(x-y-\alpha(t-\tau), t-\tau\right)G^{p}(y-\alpha\tau, \tau)dyd\tau \nonumber\\
&=\int_{0}^{t}\int_{\R}G\left(x-\alpha t-z, t-\tau\right)G^{p}(z, \tau)dzd\tau \quad (z=y-\alpha \tau) \nonumber\\
&=t\int_{0}^{1}\int_{\R}G\left(x-\alpha t-z, t(1-s)\right)G^{p}(z, ts)dzds \quad (\tau=ts) \nonumber\\
&=t^{\frac{3}{2}}\int_{0}^{1}\int_{\R}G\left(x-\alpha t-\sqrt{t}v, t(1-s)\right)G^{p}\left(\sqrt{t}v, ts\right)dvds \quad (z=\sqrt{t}v) \nonumber\\
&=t^{\frac{3}{2}}\int_{0}^{1}\int_{\R}t^{-\frac{1}{2}}G\left(\frac{x-\alpha t}{\sqrt{t}}-v, 1-s\right)t^{-\frac{p}{2}}G^{p}(v, s)dvds \nonumber\\
&=t^{-\frac{p-2}{2}}\int_{0}^{1}\left(G(1-s)*G^{p}(s)\right)\left(\frac{x-\alpha t}{\sqrt{t}}\right)ds. \label{derive-W-pre}
\end{align}
Therefore, from \eqref{derive-W-pre}, \eqref{alpha-DEF} and \eqref{asymp-func}, we can derive $W_{p}(x, t)$ as follows: 
\begin{align}
&\int_{0}^{t}\p_{x}G_{0}(t-\tau)*\left(|MG_{0}|^{p-1}MG_{0}\right)(\tau)d\tau \nonumber\\
&=\left(|M|^{p-1}M\right)t^{-\frac{p-2}{2}}\p_{x}\left(\int_{0}^{1}\left(G(1-s)*G^{p}(s)\right)\left(\frac{x-\alpha t}{\sqrt{t}}\right)ds\right) \nonumber\\
&=\left(|M|^{p-1}M\right)t^{-\frac{p-1}{2}}\frac{d}{dx}\left(\int_{0}^{1}\left(G(1-s)*G^{p}(s)\right)\left(x\right)ds\right)\biggl|_{x=\frac{x-\alpha t}{\sqrt{t}}} \nonumber\\
&=\left(|M|^{p-1}M\right)t^{-\frac{p-1}{2}}w_{p}\left(\frac{x-\alpha t}{\sqrt{t}}\right)=\left(|M|^{p-1}M\right)W_{p}(x, t). \label{derive-W}
\end{align}

Finally, we shall prove \eqref{nl-2<p<3}. From, \eqref{derive-W}, we note that the following relation holds: 
\begin{align}
&\int_{0}^{t}\p_{x}T(t-\tau)*\left(|u|^{p-1}u\right)(\tau)d\tau-\left(|M|^{p-1}M\right)W_{p}(x, t) \nonumber \\
&=\int_{0}^{t}\p_{x}T(t-\tau)*\left(|u|^{p-1}u\right)(\tau)d\tau-\int_{0}^{t}\p_{x}G_{0}(t-\tau)*\left(|MG_{0}|^{p-1}MG_{0}\right)(\tau)d\tau \nonumber \\
&=\int_{0}^{t}\p_{x}\left(T-G_{0}\right)(t-\tau)*\left(|u|^{p-1}u\right)(\tau)d\tau \nonumber \\
&\ \ \ +\int_{0}^{t}\p_{x}G_{0}(t-\tau)*\left(|u|^{p-1}u-|MG_{0}|^{p-1}MG_{0}\right)(\tau)d\tau. \label{nl-2<p<3-split}
\end{align}
Therefore, by virtue of \eqref{nl-2<p<3-split}, Proposition~\ref{prop.asymp-Duhamel-pre}, \eqref{K-DEF} and \eqref{K-limit}, 
we can conclude that \eqref{nl-2<p<3} is true. 
\end{proof}
%%%%%%%%%%%%%%%%%%%%%%%%%%%%%%%%%%%%%%%%%%%%%%%%%%%%

%%%%%%%%%%%%%%%%%%%%%%%%%%%%%%%%%%%%%%%%%%%%%%%%%%%%
\begin{proof}[\rm{\bf{End of the Proof of Theorem~\ref{thm.main}~for~$\bm{2<p<3}$}}]
We note that the following relation holds: 
\begin{align}
&u(x, t)-MG_{0}(x, t)+\left(|M|^{p-1}M\right)W_{p}(x, t) \nonumber \\
&=\left\{T(t)*u_{0}-G_{0}(t)*u_{0}\right\} +\left\{G_{0}(t)*u_{0}-MG_{0}(x, t)\right\} \nonumber \\
&\ \ \ \,-\left\{\int_{0}^{t}\p_{x}T(t-\tau)*\left(|u|^{p-1}u\right)(\tau)d\tau-\left(|M|^{p-1}M\right)W_{p}(x, t)\right\}.  \label{sol-split-2<p<3}
\end{align}
Therefore, from \eqref{sol-split-2<p<3}, Young's inequality, Lemma~\ref{lem.T-G-linear}, Propositions~\ref{prop.asymp-heat} and \ref{prop.nl-2<p<3}, 
we can conclude that the asymptotic formula \eqref{u-asymp-2<p<3} is true. This completes the proof of Theorem~\ref{thm.main} for $2<p<3$. 
\end{proof}
%%%%%%%%%%%%%%%%%%%%%%%%%%%%%%%%%%%%%%%%%%%%%%%%%%%%

%%%%%%%%%%%%%%%%%%%
\subsection{Proof of Theorem~\ref{thm.main} for $\bm{p=3}$}
%%%%%%%%%%%%%%%%%%%

Next, in this subsection, let us treat the case of $p=3$, i.e. we shall prove \eqref{u-asymp-p=3}. 
In order to prove it, we derive the following key asymptotic formula. 
The methods used in the proof of the proposition below are based on the techniques used for Proposition~4.3 in \cite{F19, K07}. 
%%%%%%%%%%%%%%%%%%%%%%%%%%%%%%%%%%%%%%%%%%%%%%%%%%%%
\begin{prop}
\label{prop.nl-p=3}
Let $p=3$. Assume that $u_{0} \in H^{1}(\R) \cap L^{1}(\R)$, $xu_{0}\in L^{1}(\R)$ and $E_{0}=\|u_{0}\|_{H^{1}}+\|u_{0}\|_{L^{1}}$ is sufficiently small. 
Then, the solution $u(x, t)$ to \eqref{VFW} satisfies 
\begin{equation}\label{nl-p=3}
\left\|\int_{0}^{t}\p_{x}T(t-\tau)*u^{3}(\tau)d\tau-\frac{M^{3}}{4\sqrt{3}\pi \mu}\left(\log t\right)\p_{x}G_{0}(\cdot, t)\right\|_{L^{q}}
\le CE_{1}t^{-\frac{1}{2}\left(1-\frac{1}{q}\right)-\frac{1}{2}}, \ \ t\ge2,
\end{equation}
for any $2\le q \le \infty$, where $T(x, t)$, $G_{0}(x, t)$ and $M$ are defined by \eqref{T-linear}, \eqref{asymp-func} and \eqref{M-m-notation}, respectively. 
Also, the constant $E_{1}$ is defined by $E_{1}=E_{0}+\|xu_{0}\|_{L^{1}}$. 
\end{prop}
%%%%%%%%%%%%%%%%%%%%%%%%%%%%%%%%%%%%%%%%%%%%%%%%%%%%
%%%%%%%%%%%%%%%%%%%%%%%%%%%%%%%%%%%%%%%%%%%%%%%%%%%%
\begin{proof}
First, we split the integral in the Duhamel term of \eqref{integral-eq} as follows: 
\begin{align}
&\int_{0}^{t}\p_{x}T(t-\tau)*u^{3}(\tau)d\tau
=\left(\int_{0}^{1}+\int_{1}^{t/2}+\int_{t/2}^{t}\right)\p_{x}T(t-\tau)*u^{3}(\tau)d\tau \nonumber \\
&=\int_{0}^{1}\p_{x}T(t-\tau)*u^{3}(\tau)d\tau+\int_{t/2}^{t}\p_{x}T(t-\tau)*u^{3}(\tau)d\tau  \nonumber \\
&\ \ \ \, +\int_{1}^{t/2}\p_{x}\left(T-G_{0}\right)(t-\tau)*u^{3}(\tau)d\tau+\int_{1}^{t/2}\p_{x}G_{0}(t-\tau)*\left(u^{3}-\left(MG_{0}\right)^{3}\right)(\tau)d\tau  \nonumber \\
&\ \ \ \, +M^{3}\int_{1}^{t/2}\p_{x}G_{0}(t-\tau)*G_{0}^{3}(\tau)d\tau  \nonumber \\
&=:R_{1}(x, t)+R_{2}(x, t)+R_{3}(x, t)+R_{4}(x, t)+L(x, t). \label{int-p=3}
\end{align}

In what follows, let us evaluate $R_{i}(x, t)$ for all $i=1, 2, 3, 4$. We start with evaluation for $R_{1}(x, t)$. 
Modifying the way to get \eqref{N1-est} and using Young's inequality, \eqref{T-basic} and \eqref{u^p-decay-Lq}, we have 
\begin{align}
\left\|R_{1}(\cdot, t)\right\|_{L^{q}} 
&\le \int_{0}^{1}\left\|\p_{x}T(\cdot, t-\tau)\right\|_{L^{q}}\left\|u^{3}(\cdot, \tau)\right\|_{L^{1}}d\tau 
\le CE_{0}\int_{0}^{1}(t-\tau)^{-\frac{1}{2}\left(1-\frac{1}{q}\right)-\frac{1}{2}}(1+\tau)^{-1}d\tau \nonumber \\
&\le CE_{0}(t-1)^{-\frac{1}{2}\left(1-\frac{1}{q}\right)-\frac{1}{2}}\int_{0}^{1}(1+\tau)^{-1}d\tau
\le CE_{0}t^{-\frac{1}{2}\left(1-\frac{1}{q}\right)-\frac{1}{2}}, \ \ t\ge2, \ 2\le q\le \infty.  \label{R1-est}
\end{align}
For $R_{2}(x, t)$, we can use the estimate \eqref{N2-est} because $R_{2}(x, t)\equiv N_{2}(x, t)$. Therefore, we obtain 
\begin{equation}
\left\|R_{2}(\cdot, t)\right\|_{L^{q}} \le CE_{0}t^{-\frac{1}{2}\left(1-\frac{1}{q}\right)-\frac{1}{2}}, \ \ t>0, \ 2\le q\le \infty. \label{R2-est}
\end{equation}
Next, we deal with $R_{3}(x, t)$. In a similar way to get \eqref{D1-est} for $p=3$, by using Young's inequality, \eqref{T-G-linear} and \eqref{u^p-decay-Lq}, we are able to see that 
\begin{align}
\left\|R_{3}(\cdot, t)\right\|_{L^{q}} 
&\le \int_{1}^{t/2}\left\|\p_{x}(T-G_{0})(\cdot, t-\tau)\right\|_{L^{q}}\left\|u^{3}(\cdot, \tau)\right\|_{L^{1}}d\tau \nonumber \\
&\le CE_{0}\int_{1}^{t/2}(t-\tau)^{-\frac{1}{2}\left(1-\frac{1}{q}\right)-1}(1+\tau)^{-1}d\tau \nonumber \\
&\le CE_{0}t^{-\frac{1}{2}\left(1-\frac{1}{q}\right)-1}\log(1+t), \ \ t\ge2, \ 2\le q\le \infty.  \label{R3-est}
\end{align}
Finally, let us treat $R_{4}(x, t)$, it follows from Young's inequality, Lemma~\ref{lem.heat-decay} and \eqref{decay-part} for $p=3$ that 
\begin{align}
\left\|R_{4}(\cdot, t)\right\|_{L^{q}} 
&\le \int_{1}^{t/2}\left\|\p_{x}G_{0}(\cdot, t-\tau)\right\|_{L^{q}}\left\|\left(u^{3}-(MG_{0})^{3}\right)(\cdot, \tau)\right\|_{L^{1}}d\tau \nonumber \\
&\le CE_{1}\int_{1}^{t/2}(t-\tau)^{-\frac{1}{2}\left(1-\frac{1}{q}\right)-\frac{1}{2}}\tau^{-\frac{3}{2}}\log(2+\tau)d\tau \nonumber \\
&\le CE_{1}t^{-\frac{1}{2}\left(1-\frac{1}{q}\right)-\frac{1}{2}}, \ \ t\ge2, \ 1\le q\le \infty.  \label{R4-est}
\end{align}

Next, we shall transform $L(x, t)$ in \eqref{int-p=3} and derive the leading term of it. Here, we recall the definitions of $G_{0}(x, t)$, $G(x, t)$ and $\alpha$, i.e. \eqref{asymp-func}, \eqref{kernel} and \eqref{alpha-DEF}, respectively. Then, by the change of valuable and the integration by parts, we have
\begin{align}
L(x, t)&=M^{3}\int_{1}^{t/2}\int_{\R}\p_{x}G\left(x-y-\alpha(t-\tau), t-\tau\right)G^{3}(y-\alpha \tau, \tau)dyd\tau \nonumber\\
&=M^{3}\int_{1}^{t/2}\int_{\R}\p_{x}G(x-\alpha t-z, t-\tau)G^{3}(z, \tau)dzd\tau \quad \left(z=y-\alpha \tau \right) \nonumber\\
&=M^{3}\int_{1}^{t/2}\p_{x}G(x-\alpha t, t-\tau)\int_{\R}G^{3}(\eta, \tau)d\eta d\tau \nonumber\\
&\ \ \ -M^{3}\int_{1}^{t/2}\int_{0}^{\infty}\p_{x}^{2}G(x-\alpha t-z, t-\tau)\int_{z}^{\infty}G^{3}(\eta, \tau)d\eta dzd\tau \nonumber\\
&\ \ \ +M^{3}\int_{1}^{t/2}\int_{-\infty}^{0}\p_{x}^{2}G(x-\alpha t-z, t-\tau)\int_{-\infty}^{z}G^{3}(\eta, \tau)d\eta dzd\tau \nonumber\\
&=:L_{0}(x, t)+L_{1}(x, t)+L_{2}(x, t). \label{L-split}
\end{align}
Now, let us evaluate $L_{1}(x, t)$ and $L_{2}(x, t)$. For the heat kernel $G(x, t)$ defined by \eqref{kernel}, we recall the following well known estimate (for the proof, see e.g. \cite{GG99}):
\begin{equation}\label{est.Gauss}
\left\|\p_{t}^{k}\p_{x}^{l}G(\cdot, t)\right\|_{L^{q}}\le Ct^{-\frac{1}{2}\left(1-\frac{1}{q}\right)-\frac{l}{2}-k}, \ \ t>0, \ 1\le q\le \infty. 
\end{equation}
Therefore, it follows from \eqref{est.Gauss} that 
\begin{align}
\left\|L_{1}(\cdot, t)\right\|_{L^{q}} 
&\le |M|^{3}\int_{1}^{t/2}\int_{0}^{\infty}\left\|\p_{x}^{2}G(\cdot-\alpha t-z, t-\tau)\right\|_{L^{q}}\int_{z}^{\infty}G^{3}(\eta, \tau)d\eta dzd\tau \nonumber\\
&\le C|M|^{3}t^{-\frac{1}{2}\left(1-\frac{1}{q}\right)-1}\int_{1}^{t/2}\int_{0}^{\infty}\int_{z}^{\infty}G^{3}(\eta, \tau)d\eta dzd\tau \nonumber\\
&=C|M|^{3}t^{-\frac{1}{2}\left(1-\frac{1}{q}\right)-1}\int_{1}^{t/2}\int_{0}^{\infty}\int_{0}^{\eta}G^{3}(\eta, \tau)dzd\eta d\tau \nonumber\\
&\le C|M|^{3}t^{-\frac{1}{2}\left(1-\frac{1}{q}\right)-1}\int_{1}^{t/2}\int_{0}^{\infty}\eta G^{3}(\eta, \tau) d\eta d\tau \nonumber\\
&\le C|M|^{3}t^{-\frac{1}{2}\left(1-\frac{1}{q}\right)-1}\int_{1}^{t/2}\tau^{-\frac{1}{2}} d\tau  \nonumber\\
&\le C|M|^{3}t^{-\frac{1}{2}\left(1-\frac{1}{q}\right)-\frac{1}{2}}, \ \ t\ge2, \ 1\le q\le \infty,  \label{L1-est}
\end{align}
where we have used the following fact: 
\[
\int_{0}^{\infty}\eta G^{3}(\eta, \tau) d\eta
=\int_{0}^{\infty}\frac{\eta}{(4 \pi \mu \tau)^{\frac{3}{2}}}\exp\left(-\frac{3\eta^{2}}{4\mu\tau}\right)d\eta
\le C\tau^{-\frac{1}{2}}. 
\]
Analogously, we can obtain the same estimate for $L_{2}(x, t)$ as follows: 
\begin{equation}\label{L2-est}
\left\|L_{2}(\cdot, t)\right\|_{L^{q}} 
\le C |M|^{3}t^{-\frac{1}{2}\left(1-\frac{1}{q}\right)-\frac{1}{2}}, \ \ t\ge2, \ 1\le q\le \infty. 
\end{equation}
Finally, we would like to treat $L_{0}(x, t)$. First, we note that
\[
\int_{\R}G^{3}(\eta, \tau)d\eta=\int_{\R}\frac{1}{(4 \pi \mu \tau)^{\frac{3}{2}}}\exp\left(-\frac{3\eta^{2}}{4\mu\tau}\right)d\eta=\frac{\tau^{-1}}{4\sqrt{3}\pi \mu}.
\]
Therefore, we can see that 
\begin{align}
L_{0}(x, t)
&=M^{3}\int_{1}^{t/2}\p_{x}G(x-\alpha t, t-\tau)\int_{\R}G^{3}(\eta, \tau)d\eta d\tau
=\frac{M^{3}}{4\sqrt{3}\pi \mu}\int_{1}^{t/2}\p_{x}G(x-\alpha t, t-\tau)\tau^{-1}d\tau \nonumber\\
&=\frac{M^{3}}{4\sqrt{3}\pi \mu}\int_{1}^{t/2}\p_{x}\left(G(x-\alpha t, t-\tau)-G(x-\alpha t, t)\right)\tau^{-1}d\tau
+\frac{M^{3}}{4\sqrt{3}\pi \mu}\p_{x}G(x-\alpha t, t)\log\frac{t}{2} \nonumber\\
&=\frac{M^{3}}{4\sqrt{3}\pi \mu}\int_{1}^{t/2}\p_{x}\left(G(x-\alpha t, t-\tau)-G(x-\alpha t, t)\right)\tau^{-1}d\tau \nonumber \\
&\ \ \ \,-\frac{M^{3}\log2}{4\sqrt{3}\pi \mu}\p_{x}G_{0}(x, t)
+\frac{M^{3}}{4\sqrt{3}\pi \mu}\left(\log t\right)\p_{x}G_{0}(x, t) \nonumber \\
&=:L_{0.1}(x, t)+L_{0.2}(x, t)+\frac{M^{3}}{4\sqrt{3}\pi \mu}\left(\log t\right)\p_{x}G_{0}(x, t). 
\label{L0-split}
\end{align}
To evaluate $L_{0.1}(x, t)$, we use the following fact: 
\begin{equation}\label{t-mean}
\left\|\p_{x}\left(G(\cdot-\alpha t, t-\tau)-G(\cdot-\alpha t, t)\right)\right\|_{L^{q}}\le C\tau(t-\tau)^{-\frac{1}{2}\left(1-\frac{1}{q}\right)-\frac{3}{2}}, \ \ t>\tau, \ 1\le q\le \infty,
\end{equation}
which comes from the mean value theorem
\[
\p_{x}G(x-\alpha t, t-\tau)-\p_{x}G(x-\alpha t, t)=-\tau\int_{0}^{1}\left(\p_{t}\p_{x}G\right)(x-\alpha t, t-\theta \tau)d\theta
\]
and \eqref{est.Gauss}. Therefore, it follows from \eqref{t-mean} that 
\begin{align}
\left\|L_{0.1}(\cdot, t)\right\|_{L^{q}}
&\le C|M|^{3}\int_{1}^{t/2}\tau(t-\tau)^{-\frac{1}{2}\left(1-\frac{1}{q}\right)-\frac{3}{2}}\tau^{-1} d\tau \nonumber \\
&\le C|M|^{3}t^{-\frac{1}{2}\left(1-\frac{1}{q}\right)-\frac{1}{2}}, \ \ t\ge2, \ 1\le q\le \infty. \label{L0.1-est}
\end{align}
On the other hand, it directly follows from Lemma~\ref{lem.heat-decay} that 
\begin{equation}\label{L0.2-est}
\left\|L_{0.2}(\cdot, t)\right\|_{L^{q}}
\le C|M|^{3}t^{-\frac{1}{2}\left(1-\frac{1}{q}\right)-\frac{1}{2}}, \ \ t\ge2, \ 1\le q\le \infty. 
\end{equation}

Eventually, combining \eqref{int-p=3} through \eqref{L-split}, \eqref{L1-est} through \eqref{L0-split}, \eqref{L0.1-est} and \eqref{L0.2-est}, we arrive at
\begin{align*}
&\left\|\int_{0}^{t}\p_{x}T(t-\tau)*u^{3}(\tau)d\tau-\frac{M^{3}}{4\sqrt{3}\pi \mu}\left(\log t\right)\p_{x}G_{0}(\cdot, t)\right\|_{L^{q}}\\
&\le \sum_{i=1}^{4}\left\|R_{i}(\cdot, t)\right\|_{L^{q}}+\sum_{j=1}^{2}\left\|L_{j}(\cdot, t)\right\|_{L^{q}}+\sum_{k=1}^{2}\left\|L_{0.k}(\cdot, t)\right\|_{L^{q}}
\le CE_{1}t^{-\frac{1}{2}\left(1-\frac{1}{q}\right)-\frac{1}{2}}, \ \ t\ge2, \ 2\le q\le \infty. 
\end{align*}
Therefore, we obtain the desired estimate \eqref{nl-p=3}. This completes the proof. 
\end{proof}
%%%%%%%%%%%%%%%%%%%%%%%%%%%%%%%%%%%%%%%%%%%%%%%%%%%%

Now, we note that the following relation holds: 
\begin{align}
&u(x, t)-MG_{0}(x, t)+\frac{M^{3}}{4\sqrt{3}\pi \mu}\left(\log t\right)\p_{x}G_{0}(x, t) \nonumber \\
&=\left\{T(t)*u_{0}-G_{0}(t)*u_{0}\right\} +\left\{G_{0}(t)*u_{0}-MG_{0}(x, t)\right\} \nonumber \\
&\ \ \ \,-\left\{\int_{0}^{t}\p_{x}T(t-\tau)*\left(|u|^{p-1}u\right)(\tau)d\tau-\frac{M^{3}}{4\sqrt{3}\pi \mu}\left(\log t\right)\p_{x}G_{0}(x, t)\right\}.  \label{sol-split-p=3}
\end{align}
Therefore, by virtue of \eqref{sol-split-p=3}, Young's inequality, Lemma~\ref{lem.T-G-linear}, Propositions~\ref{prop.asymp-heat} and \ref{prop.nl-p=3}, 
we can immediately conclude that the following asymptotic formula is true: 
%%%%%%%%%%%%%%%%%%%%%%%%%%%%%%%%%%%%%%%%%%%%%%%%%%%%
\begin{thm}
\label{thm.nl-p=3}
Let $p=3$. Assume that $u_{0} \in H^{1}(\R) \cap L^{1}(\R)$, $xu_{0}\in L^{1}(\R)$ and $E_{0}=\|u_{0}\|_{H^{1}}+\|u_{0}\|_{L^{1}}$ is sufficiently small. 
Then, the solution $u(x, t)$ to \eqref{VFW} satisfies 
\begin{equation*}
\left\|u(\cdot, t)-MG_{0}(\cdot, t)+\frac{M^{3}}{4\sqrt{3}\pi \mu}\left(\log t\right)\p_{x}G_{0}(\cdot, t)\right\|_{L^{q}}
\le CE_{1}t^{-\frac{1}{2}\left(1-\frac{1}{q}\right)-\frac{1}{2}}, \ \ t\ge2,
\end{equation*}
for any $2\le q \le \infty$, where $G_{0}(x, t)$ and $M$ are defined by \eqref{asymp-func} and \eqref{M-m-notation}, respectively. 
Also, the constant $E_{1}$ is defined by $E_{1}=E_{0}+\|xu_{0}\|_{L^{1}}$. 
\end{thm}
%%%%%%%%%%%%%%%%%%%%%%%%%%%%%%%%%%%%%%%%%%%%%%%%%%%%

%%%%%%%%%%%%%%%%%%%%%%%%%%%%%%%%%%%%%%%%%%%%%%%%%%%%
\begin{proof}[\rm{\bf{End of the Proof of Theorem~\ref{thm.main}~for~$\bm{p=3}$}}]
The desired result \eqref{u-asymp-p=3} can be easily obtained from the above Theorem~\ref{thm.nl-p=3}. 
This completes the proof of Theorem~\ref{thm.main} for $p=3$. 
\end{proof}
%%%%%%%%%%%%%%%%%%%%%%%%%%%%%%%%%%%%%%%%%%%%%%%%%%%%

%%%%%%%%%%%%%%%%%%%
\subsection{Proof of Theorem~\ref{thm.main} for $\bm{p>3}$}
%%%%%%%%%%%%%%%%%%%

Finally in this subsection, we complete the proof of Theorem~\ref{thm.main} for $p>3$, i.e. we shall prove \eqref{u-asymp-p>3}. 
In order to show it, we need to analyze the linear part of the solution to \eqref{VFW} in more details. 
To doing that, let us further transform the dispersion term in \eqref{VFW}. Now, recalling \eqref{dispersion} and noticing that 
\begin{align}
\frac{2B}{b}\p_{x}u+\frac{2B}{b}(b^{2}-\p_{x}^{2})^{-1}\p_{x}^{3}u
=\frac{2B}{b}\p_{x}u+\frac{2B}{b^{3}}\p_{x}^{3}u+\frac{2B}{b^{3}}(b^{2}-\p_{x}^{2})^{-1}\p_{x}^{5}u. \label{dispersion-2}
\end{align}
Therefore, the integral kernel $T(x, t)$ is defined by \eqref{T-linear} can be rewritten by 
\begin{align}\label{T-linear-2}
\begin{split}
T(x, t)=\frac{1}{\sqrt{2\pi}}\mathcal{F}^{-1}\left[\exp\left(-\mu t\xi^{2}-\frac{i2Bt\xi}{b}+\frac{i2Bt\xi^{3}}{b^{3}}-\frac{i2Bt\xi^{5}}{b^{3}(b^{2}+\xi^{2})}\right)\right](x).
\end{split}
\end{align}
To prove \eqref{u-asymp-p>3}, we need to find the asymptotic profile of both the linear part and the Duhamel part of \eqref{integral-eq}. 
First, we shall explain about the asymptotic analysis for the linear part. 
The following proposition is a key to derive the leading term of \eqref{integral-eq}. For some related results to this formula, see e.g. \cite{FI21, K99}.
%%%%%%%%%%%%%%%%%%%%%%%%%%%%%%%%%%%%%%%%%%%%%%%%%%%%
\begin{prop}
\label{prop.asymp-linear}
Let $l$ be a non-negative integer and $2\le q\le \infty$. Then, we have 
\begin{equation}\label{asymp-linear}
\left\|\p_{x}^{l}\left(T(\cdot, t)-G_{0}(\cdot, t)+\frac{2B}{b^{3}}t\p_{x}^{3}G_{0}(\cdot, t)\right)\right\|_{L^{q}}\le Ct^{-\frac{1}{2}\left(1-\frac{1}{q}\right)-1-\frac{l}{2}}\left(1+t^{-\frac{1}{2}}\right), \ \ t>0,
\end{equation}
where $T(x, t)$ and $G_{0}(x, t)$ are defined by \eqref{T-linear} and \eqref{asymp-func}, respectively. 
\end{prop}
%%%%%%%%%%%%%%%%%%%%%%%%%%%%%%%%%%%%%%%%%%%%%%%%%%%%
%%%%%%%%%%%%%%%%%%%%%%%%%%%%%%%%%%%%%%%%%%%%%%%%%%%%
\begin{proof}
First, applying the Fourier transform to $T(x, t)$, then from \eqref{T-linear-2} and Taylor's theorem, there exist $\theta_{0}, \theta_{1} \in (0, 1)$ such that the following relation holds:
\begin{align}
\hat{T}(\xi, t)&=\frac{1}{\sqrt{2\pi}}\exp\left(-\mu t\xi^{2}-\frac{i2Bt\xi}{b}+\frac{i2Bt\xi^{3}}{b^{3}}\right)\left\{1-\frac{i2Bt\xi^{5}}{b^{3}(b^{2}+\xi^{2})}\exp\left(-\frac{i\theta_{1} 2Bt\xi^{5}}{b^{3}(b^{2}+\xi^{2})}\right)\right\} \nonumber \\
&=\frac{1}{\sqrt{2\pi}}\exp\left(-\mu t\xi^{2}-\frac{i2Bt\xi}{b}\right)\left\{1+\frac{i2Bt\xi^{3}}{b^{3}}+\frac{1}{2}\left(\frac{i2Bt\xi^{3}}{b^{3}}\right)^{2}\exp\left(\frac{i\theta_{0}2Bt\xi^{3}}{b^{3}}\right)\right\} \nonumber \\
&\ \ \ \ -\frac{i2Bt\xi^{5}}{\sqrt{2\pi}b^{3}(b^{2}+\xi^{2})}\exp\left(-\mu t\xi^{2}-\frac{i2Bt\xi}{b}+\frac{i2Bt\xi^{3}}{b^{3}}-\frac{i\theta_{1} 2Bt\xi^{5}}{b^{3}(b^{2}+\xi^{2})}\right) \nonumber \\
&=\frac{1}{\sqrt{2\pi}}\exp\left(-\mu t\xi^{2}-\frac{i2Bt\xi}{b}\right)+\frac{i2Bt\xi^{3}}{b^{3}}\frac{1}{\sqrt{2\pi}}\exp\left(-\mu t\xi^{2}-\frac{i2Bt\xi}{b}\right)+R(\xi, t) \nonumber \\
&=\hat{G}_{0}(\xi, t)-\frac{2B}{b^{3}}t(i\xi)^{3}\hat{G}_{0}(\xi, t)+R(\xi, t), \ \ \xi \in \R, \ t>0, \label{T-hat}
\end{align}
where the remainder term $R(\xi, t)$ is defined by 
\begin{align*}
R(\xi, t):=&-\frac{2B^{2}t^{2}\xi^{6}}{\sqrt{2\pi}b^{6}}\exp\left(-\mu t\xi^{2}-\frac{i2Bt\xi}{b}+\frac{i\theta_{0}2Bt\xi^{3}}{b^{3}}\right)\\
&-\frac{i2Bt\xi^{5}}{\sqrt{2\pi}b^{3}(b^{2}+\xi^{2})}\exp\left(-\mu t\xi^{2}-\frac{i2Bt\xi}{b}+\frac{i2Bt\xi^{3}}{b^{3}}-\frac{i\theta_{1} 2Bt\xi^{5}}{b^{3}(b^{2}+\xi^{2})}\right), \ \ \xi \in \R, \ t>0. 
\end{align*}
Here, we note that the following estimate holds: 
\begin{equation}\label{Remainder} 
|R(\xi, t)| \le C\left(t^{2}\xi^{6}+t\xi^{5}\right)e^{-\mu t\xi^{2}}, \ \ \xi \in \R, \ t>0. 
\end{equation}
Therefore, by using the Plancherel theorem, from \eqref{T-hat} and \eqref{Remainder}, we have 
\begin{align*}
&\left\|\p_{x}^{l}\left(T(\cdot, t)-G_{0}(\cdot, t)+\frac{2B}{b^{3}}t\p_{x}^{3}G_{0}(\cdot, t)\right)\right\|_{L^{2}}^{2} \\
&=\left\|(i\xi)^{l}\left(\hat{T}(\xi, t)-\hat{G}_{0}(\xi, t)+\frac{2B}{b^{3}}t(i\xi)^{3}\hat{G}_{0}(\xi, t)\right)\right\|_{L^{2}}^{2}
=\left\|(i\xi)^{l}R(\xi, t)\right\|_{L^{2}}^{2} \\ 
&\le C \int_{\R} \xi^{2l}\left(t^{2}\xi^{6}+t\xi^{5}\right)^{2}e^{-2\mu t\xi^{2}}d\xi
\le C \int_{\R} \left(t^{4}\xi^{2(l+6)}+t^{2}\xi^{2(l+5)}\right)e^{-2\mu t\xi^{2}}d\xi \\
&\le C\left(t^{-\frac{5}{2}-l}+t^{-\frac{7}{2}-l}\right)=Ct^{-\frac{5}{2}-l}\left(1+t^{-1}\right), \ \ t>0.
\end{align*}
Thus, we have the $L^{2}$-decay estimate:
\begin{equation}\label{linear-ap-L2}
\left\|\p_{x}^{l}\left(T(\cdot, t)-G_{0}(\cdot, t)+\frac{2B}{b^{3}}t\p_{x}^{3}G_{0}(\cdot, t)\right)\right\|_{L^{2}}
\le Ct^{-\frac{5}{4}-\frac{l}{2}}\left(1+t^{-\frac{1}{2}}\right), \ \ t>0.
\end{equation}

For the $L^{\infty}$-decay estimate, from the Sobolev inequality \eqref{Sobolev-ineq}, we can see that 
\begin{equation}\label{linear-ap-Linf}
\left\|\p_{x}^{l}\left(T(\cdot, t)-G_{0}(\cdot, t)+\frac{2B}{b^{3}}t\p_{x}^{3}G_{0}(\cdot, t)\right)\right\|_{L^{\infty}}
\le Ct^{-\frac{3}{2}-\frac{l}{2}}\left(1+t^{-\frac{1}{2}}\right), \ \ t>0.
\end{equation}
In addition, the desired estimate \eqref{asymp-linear} for $2<q<\infty$ can be easily obtained by using \eqref{linear-ap-L2}, \eqref{linear-ap-Linf} and the interpolation inequality \eqref{interpolation}, in the same way to get \eqref{u-Lq}. This completes the proof. 
\end{proof}
%%%%%%%%%%%%%%%%%%%%%%%%%%%%%%%%%%%%%%%%%%%%%%%%%%%%

By using Young's inequality and Proposition~\ref{prop.asymp-linear}, we immediately have the following formula: 
%%%%%%%%%%%%%%%%%%%%%%%%%%%%%%%%%%%%%%%%%%%%%%%%%%%%
\begin{cor}
\label{cor.asymp-linear}
Let $l$ be a non-negative integer and $2\le q\le \infty$. 
Suppose $u_{0}\in L^{1}(\R)$. Then, we have
\begin{equation}\label{sol-linear}
\left\|\p_{x}^{l}\left(T(t)*u_{0}-G_{0}(t)*u_{0}+\frac{2B}{b^{3}}t\p_{x}^{3}G_{0}(t)*u_{0}\right)\right\|_{L^{q}}\le C\|u_{0}\|_{L^{1}}t^{-\frac{1}{2}\left(1-\frac{1}{q}\right)-1-\frac{l}{2}}\left(1+t^{-\frac{1}{2}}\right), \ \ t>0,
\end{equation}
where $T(x, t)$ and $G_{0}(x, t)$ are defined by \eqref{T-linear} and \eqref{asymp-func}, respectively. 
\end{cor}
%%%%%%%%%%%%%%%%%%%%%%%%%%%%%%%%%%%%%%%%%%%%%%%%%%%%

Next, we shall introduce a key proposition to derive the leading term of the Duhamel term in \eqref{integral-eq}, in the case of $p>3$. The method for the proof of the following proposition is based on the technique used in the proof of Lemma~6.2 in \cite{FI21}. 
%%%%%%%%%%%%%%%%%%%%%%%%%%%%%%%%%%%%%%%%%%%%%%%%%%%%
\begin{prop}
\label{prop.asymp-Duhamel}
Let $p>3$. Assume that $u_{0} \in H^{1}(\R) \cap L^{1}(\R)$ and $E_{0}=\|u_{0}\|_{H^{1}}+\|u_{0}\|_{L^{1}}$ is sufficiently small. 
Then, the solution $u(x, t)$ to \eqref{VFW} satisfies  \begin{equation}\label{asymp-nonlinear}
\lim_{t\to \infty}t^{\frac{1}{2}\left(1-\frac{1}{q}\right)+\frac{1}{2}}\left\|\int_{0}^{t}\p_{x}T(t-\tau)*\left(|u|^{p-1}u\right)(\tau)d\tau-\mathcal{M}\p_{x}G_{0}(\cdot, t)\right\|_{L^{q}}=0, 
\end{equation}
for any $2\le q \le \infty$, where $T(x, t)$, $G_{0}(x, t)$ and $\mathcal{M}$ are defined by \eqref{T-linear}, \eqref{asymp-func} and \eqref{M-m-notation}, respectively. 
\end{prop}
%%%%%%%%%%%%%%%%%%%%%%%%%%%%%%%%%%%%%%%%%%%%%%%%%%%%
%%%%%%%%%%%%%%%%%%%%%%%%%%%%%%%%%%%%%%%%%%%%%%%%%%%%
\begin{proof}
By virtue of Proposition~\ref{prop.asymp-Duhamel-pre}, it is sufficient to show the following formula: 
\begin{equation}\label{asymp-nonlinear-G0}
\lim_{t\to \infty}t^{\frac{1}{2}\left(1-\frac{1}{q}\right)+\frac{1}{2}}\left\|\int_{0}^{t}\p_{x}G_{0}(t-\tau)*\left(|u|^{p-1}u\right)(\tau)d\tau-\mathcal{M}\p_{x}G_{0}(\cdot, t)\right\|_{L^{q}}=0, \ \ 2\le q\le \infty.
\end{equation}
In what follows, we shall prove \eqref{asymp-nonlinear-G0}. First, from the definition of \eqref{M-m-notation}, we have 
\begin{align}
&\int_{0}^{t}\p_{x}G_{0}(t-\tau)*\left(|u|^{p-1}u\right)(\tau)d\tau-\mathcal{M}\p_{x}G_{0}(x, t) \nonumber \\
&=\int_{0}^{t}\int_{\R}\p_{x}G_{0}(x-y, t-\tau)\left(|u|^{p-1}u\right)(y, \tau)dyd\tau-\left(\int_{0}^{\infty}\int_{\R}\left(|u|^{p-1}u\right)(y, \tau)dyd\tau\right)\p_{x}G_{0}(x, t) \nonumber \\
&=\int_{0}^{t}\int_{\R}\left(\p_{x}G_{0}(x-y, t-\tau)-\p_{x}G_{0}(x, t)\right)\left(|u|^{p-1}u\right)(y, \tau)dyd\tau \nonumber \\
&\ \ \ \ -\left(\int_{t}^{\infty}\int_{\R}\left(|u|^{p-1}u\right)(y, \tau)dyd\tau\right)\p_{x}G_{0}(x, t) 
=:X(x, t)+Y(x, t). \label{X-Y}
\end{align}

Next, let us evaluate $X(x, t)$. Before doing that, for the latter sake, we shall rewrite $G_{0}(x, t)$. Now, recalling \eqref{asymp-func} and \eqref{alpha-DEF}, we can see that 
\begin{align*}
G_{0}(x-y, t-\tau)=G(x-y-\alpha(t-\tau), t-\tau), \quad G_{0}(x, t)=G(x-\alpha t, t), \quad \alpha=\frac{2B}{b}.
\end{align*}
Here, we take small $\e>0$. By using the change of variable and splitting the integral, we get 
\begin{align}
X(x, t)&=\int_{0}^{t}\int_{\R}\left(\p_{x}G_{0}(x-y, t-\tau)-\p_{x}G_{0}(x, t)\right)\left(|u|^{p-1}u\right)(y, \tau)dyd\tau \nonumber \\
&=\int_{0}^{t}\int_{\R}\left(\p_{x}G(x-y-\alpha(t-\tau), t-\tau)-\p_{x}G(x-\alpha t, t)\right)\left(|u|^{p-1}u\right)(y, \tau)dyd\tau \nonumber\\
&=\int_{0}^{t}\int_{\R}\left(\p_{x}G(x-\alpha t-z, t-\tau)-\p_{x}G(x-\alpha t, t)\right)\left(|u|^{p-1}u\right)(z+\alpha \tau, \tau)dzd\tau \nonumber\\
&=\int_{\e t/2}^{t}\int_{\R}\left(\p_{x}G(x-\alpha t-z, t-\tau)-\p_{x}G(x-\alpha t, t)\right)\left(|u|^{p-1}u\right)(z+\alpha \tau, \tau)dzd\tau \nonumber\\
&\ \ \ \ +\int_{0}^{\e t/2}\int_{|z|\ge \e \sqrt{t}}\left(\p_{x}G(x-\alpha t-z, t-\tau)-\p_{x}G(x-\alpha t, t)\right)\left(|u|^{p-1}u\right)(z+\alpha \tau, \tau)dzd\tau \nonumber\\
&\ \ \ \ +\int_{0}^{\e t/2}\int_{|z|\le \e \sqrt{t}}\left(\p_{x}G(x-\alpha t-z, t-\tau)-\p_{x}G(x-\alpha t, t)\right)\left(|u|^{p-1}u\right)(z+\alpha \tau, \tau)dzd\tau \nonumber\\
&=:X_{1}(x, t)+X_{2}(x, t)+X_{3}(x, t). \label{X-split} 
\end{align}
Now, let us evaluate $X_{1}(x, t)$ to $X_{3}(x, t)$. First for $X_{1}(x, t)$, from Young's inequality, \eqref{est.Gauss}, \eqref{u^p-decay-Lq} and \eqref{d_u^p-decay-Lr}, we obtain 
\begin{align}
\left\|X_{1}(\cdot, t)\right\|_{L^{q}}
&= \left\|\int_{\e t/2}^{t}\int_{\R}\left(\p_{x}G(\cdot-\alpha t-z, t-\tau)-\p_{x}G(\cdot-\alpha t, t)\right)\left(|u|^{p-1}u\right)(z+\alpha \tau, \tau)dzd\tau\right\|_{L^{q}} \nonumber \\
&\le \left\|\int_{\e t/2}^{t}\int_{\R}\p_{x}G(\cdot-z, t-\tau)\left(|u|^{p-1}u\right)(z+\alpha \tau, \tau)dzd\tau\right\|_{L^{q}} \nonumber\\
&\ \ \ \ +\left\|\p_{x}G(\cdot, t)\right\|_{L^{q}}\int_{\e t/2}^{t}\int_{\R}\left|\left(|u|^{p-1}u\right)(z+\alpha \tau, \tau)\right|dzd\tau \nonumber\\
&\le \int_{\e t/2}^{t}\left\|G(\cdot, t-\tau)\right\|_{L^{2}}\left\|\p_{x}\left(|u|^{p-1}u\right)(\cdot+\alpha \tau, \tau)\right\|_{L^{r}}d\tau \ \ \left(\frac{1}{q}+1=\frac{1}{2}+\frac{1}{r}\right) \nonumber\\
&\ \ \ \ +\left\|\p_{x}G(\cdot, t)\right\|_{L^{q}}\int_{\e t/2}^{t}\left\|\left(|u|^{p-1}u\right)(\cdot+\alpha \tau, \tau)\right\|_{L^{1}}d\tau \nonumber\\
&\le CE_{0}\int_{\e t/2}^{t}(t-\tau)^{-\frac{1}{4}}(1+\tau)^{-\frac{1}{2}\left(\frac{1}{2}-\frac{1}{q}\right)-\frac{p}{2}}d\tau
+CE_{0}t^{-\frac{1}{2}\left(1-\frac{1}{q}\right)-\frac{1}{2}}\int_{\e t/2}^{t}(1+\tau)^{-\frac{p-1}{2}}d\tau  \nonumber\\
&\le CE_{0}\e^{-\frac{1}{4}+\frac{1}{2q}-\frac{p}{2}}\left(1-\frac{\e}{2}\right)^{\frac{3}{4}}t^{-\frac{1}{2}\left(1-\frac{1}{q}\right)-\frac{p-2}{2}} \nonumber \\
&\ \ \ \ +CE_{0}\e^{-\frac{p-3}{2}}t^{-\frac{1}{2}\left(1-\frac{1}{q}\right)-\frac{p-2}{2}}, \ \ t>0, \ 2\le q\le \infty. \label{X1}
\end{align}
Next, let us treat $X_{2}(x, t)$. Similarly as \eqref{X1}, we can easily have 
\begin{align}
\|X_{2}(\cdot, t)\|_{L^{q}}
&\le \int_{0}^{\e t/2}\int_{|z|\ge \e\sqrt{t}}\left(\left\|\p_{x}G(\cdot-\alpha t-z, t-\tau)\right\|_{L^{q}}+\left\|\p_{x}G(\cdot-\alpha t, t-\tau)\right\|_{L^{q}}\right)\nonumber \\
&\ \ \ \ \times \left|\left(|u|^{p-1}u\right)(z+\alpha \tau, \tau)\right|dzd\tau
\le Ct^{-\frac{1}{2}\left(1-\frac{1}{q}\right)-\frac{1}{2}}Z(t), \ \ t>0, \ 2\le q\le \infty, \label{X2}
\end{align}
where $Z(t)$ is defined by 
\[Z(t):=\int_{0}^{\e t/2}\int_{|z|\ge \e\sqrt{t}}\left|\left(|u|^{p-1}u\right)(z+\alpha \tau, \tau)\right|dzd\tau.\]
In addition, applying Lebesgue's dominated convergence theorem, we are able to see 
\begin{equation}\label{Z}
\lim_{t\to \infty}Z(t)=0, 
\end{equation}
because it follows from \eqref{u^p-decay-Lq} that  
\begin{align}
|\mathcal{M}|
&\le \int_{0}^{\infty}\int_{\R}\left|\left(|u|^{p-1}u\right)(y, \tau)\right|dyd\tau
=\int_{0}^{\infty}\left\|\left(|u|^{p-1}u\right)(\cdot, \tau)\right\|_{L^{1}}d\tau  \nonumber \\
&\le CE_{0}\int_{0}^{\infty}(1+\tau)^{-\frac{p-1}{2}}d\tau
\le CE_{0}<\infty, \ \ p>3. \label{mathM-ok}
\end{align}
Finally, we shall treat $X_{3}(x, t)$. By using the mean value theorem and \eqref{est.Gauss}, we have 
\begin{align*}
&\left\|\p_{x}G(\cdot-\alpha t-z, t-\tau)-\p_{x}G(\cdot-\alpha t, t)\right\|_{L^{q}}
=\left\|\p_{x}G(\cdot-z, t-\tau)-\p_{x}G(\cdot, t)\right\|_{L^{q}} \\
&\le \left\|\p_{x}G(\cdot-z, t-\tau)-\p_{x}G(\cdot, t-\tau)\right\|_{L^{q}} +\left\|\p_{x}G(\cdot, t-\tau)-\p_{x}G(\cdot, t)\right\|_{L^{q}} \\
&\le C(t-\tau)^{-\frac{1}{2}\left(1-\frac{1}{q}\right)-1}|z|+C(t-\tau)^{-\frac{1}{2}\left(1-\frac{1}{q}\right)-\frac{3}{2}}\tau  \\
&\le C\e t^{-\frac{1}{2}\left(1-\frac{1}{q}\right)-\frac{1}{2}}, \ \ |z|\le \e \sqrt{t}, \ 0\le \tau \le \e t/2.
\end{align*}
Thus, combining the fact \eqref{mathM-ok} and the above estimate, we obtain  
\begin{align}
\left\|X_{3}(\cdot, t)\right\|_{L^{q}}
&\le \int_{0}^{\e t/2}\int_{|z|\le \e\sqrt{t}}\left\|\p_{x}G(\cdot-\alpha t-z, t-\tau)-\p_{x}G(\cdot-\alpha t, t)\right\|_{L^{q}} \nonumber \\
&\ \ \ \ \times\left|\left(|u|^{p-1}u\right)(z+\alpha \tau, \tau)\right|dzd\tau
\le C\e t^{-\frac{1}{2}\left(1-\frac{1}{q}\right)-\frac{1}{2}}, \ \ t>0, \ 2\le q\le \infty. \label{X3}
\end{align}

On the other hand, from \eqref{u^p-decay-Lq}, we can easily have 
\begin{align}
\left\|Y(\cdot, t)\right\|_{L^{q}}
&\le \left\|\p_{x}G_{0}(\cdot, t)\right\|_{L^{q}}\int_{t}^{\infty}\int_{\R}\left|\left(|u|^{p-1}u\right)(y, \tau)\right|dyd\tau 
=\left\|\p_{x}G_{0}(\cdot, t)\right\|_{L^{q}}\int_{t}^{\infty}\left\|\left(|u|^{p-1}u\right)(\cdot, \tau)\right\|_{L^{1}}d\tau \nonumber \\
&\le CE_{0}t^{-\frac{1}{2}\left(1-\frac{1}{q}\right)-\frac{1}{2}}\int_{t}^{\infty}(1+\tau)^{-\frac{p-1}{2}}d\tau 
\le CE_{0}t^{-\frac{1}{2}\left(1-\frac{1}{q}\right)-\frac{p-2}{2}}, \ \ t>0, \ 2\le q\le \infty. \label{Y-est}
\end{align} 

Summarizing up \eqref{X-split}, \eqref{X1} through \eqref{Y-est}, for $p>3$, we eventually arrive at 
\[\limsup_{t\to \infty}t^{\frac{1}{2}\left(1-\frac{1}{q}\right)+\frac{1}{2}}\left\|\int_{0}^{t}\p_{x}G_{0}(t-\tau)*\left(|u|^{p-1}u\right)(\tau)d\tau-\mathcal{M}\p_{x}G_{0}(\cdot, t)\right\|_{L^{q}}\le C\e, \ \ 2\le q\le \infty.\]
Thus, we obtain \eqref{asymp-nonlinear-G0} because $\e>0$ can be chosen arbitrarily small. 
Therefore, combining Proposition~\ref{prop.asymp-Duhamel-pre} and \eqref{asymp-nonlinear-G0}, 
we arrive at the desired result \eqref{asymp-nonlinear}. This completes the proof. 
\end{proof}
%%%%%%%%%%%%%%%%%%%%%%%%%%%%%%%%%%%%%%%%%%%%%%%%%%%%

%%%%%%%%%%%%%%%%%%%%%%%%%%%%%%%%%%%%%%%%%%%%%%%%%%%%
\begin{proof}[\rm{\bf{End of the Proof of Theorem~\ref{thm.main}~for~$\bm{p>3}$}}]
We note that the following relation holds: 
\begin{align}
&u(x, t)-MG_{0}(x, t)+\left(m+\mathcal{M}\right)\p_{x}G_{0}(x, t)+\frac{2BM}{b^{3}}t\p_{x}^{3}G_{0}(x, t) \nonumber \\
&=\left\{T(t)*u_{0}-G_{0}(t)*u_{0}+\frac{2B}{b^{3}}t\p_{x}^{3}G_{0}(t)*u_{0}\right\} \nonumber \\
&\ \ \ \,+\left\{G_{0}(t)*u_{0}-MG_{0}(x, t)+m\p_{x}G_{0}(x, t)\right\}
-\frac{2B}{b^{3}}t\p_{x}^{3}\left(G_{0}(t)*u_{0}-MG_{0}(x, t)\right) \nonumber \\
&\ \ \ \,-\left\{\int_{0}^{t}\p_{x}T(t-\tau)*\left(|u|^{p-1}u\right)(\tau)d\tau-\mathcal{M}\p_{x}G_{0}(x, t)\right\}.  \label{sol-split-p>3}
\end{align}
Therefore, from \eqref{sol-split-p>3}, Corollary~\ref{cor.asymp-linear}, Propositions~\ref{prop.asymp-heat} and \ref{prop.asymp-Duhamel}, 
we can conclude that the asymptotic formula \eqref{u-asymp-p>3} is true. This completes the proof of Theorem~\ref{thm.main} for $p>3$. 
\end{proof}
%%%%%%%%%%%%%%%%%%%%%%%%%%%%%%%%%%%%%%%%%%%%%%%%%%%%

%%%%%%%%%%%%%%%%%%%
\section*{Acknowledgment}
%%%%%%%%%%%%%%%%%%%
This study is supported by Grant-in-Aid for Young Scientists Research No.22K13939, Japan Society for the Promotion of Science. 

%%%%%%%%%%%%%%%%%%%
% References
%%%%%%%%%%%%%%%%%%%

%%%%%%%%%%%%%%%%%%%%%%%%%%%%%%%%%%%%%%%%%%%%%%%%%%%%

%%%%%%%%%%%%%%%%%%%%%%%%%%%%%%%%%%%%%%%%%%%%%%%%%%%%

%%%%%%%%%%%%%%%%%%%%%%%  AUTHORS  %%%%%%%%%%%%%%%%%%%%%%
\bigskip
\par\noindent
\begin{flushleft}Ikki Fukuda\\
Division of Mathematics and Physics, \\
Faculty of Engineering, \\
Shinshu University\\
4-17-1, Wakasato, Nagano, 380-8553, JAPAN\\
E-mail: i\_fukuda@shinshu-u.ac.jp
\end{flushleft}
%%%%%%%%%%%%%%%%%%%%%%%%%%%%%%%%%%%%%%%%%%%%%%%%%%%%

%%%%%%%%%%%%%%%%%%%%%%%%%%%%%%%%%%%%%%%%%%%%%%%%%%%%
\end{document}